\documentclass[12pt,a4paper]{amsart}
\usepackage{amsmath,amsthm,amssymb,stmaryrd,latexsym,a4wide,epic,eepic,bbm,mathrsfs,amsfonts,cases,yfonts,indentfirst,bbold}
\usepackage{dsfont,color,mathdots}

\newtheorem{theorem}{Theorem}
\newtheorem{lemma}[theorem]{Lemma}
\newtheorem{corollary}[theorem]{Corollary}
\newtheorem{proposition}[theorem]{Proposition}
\newtheorem{example}[theorem]{Example}

\newtheorem{remark}[theorem]{Remark}
\usepackage[all]{xy}
\usepackage[active]{srcltx}
\usepackage[parfill]{parskip}
\usepackage{enumerate}
\numberwithin{equation}{section}

\newcommand{\tto}{\twoheadrightarrow}

\begin{document}

\title[Simple transitive 2-representations and Drinfeld center]{simple
transitive $2$-representations and Drinfeld center for some finitary $2$-categories}

\author{Xiaoting Zhang}

\begin{abstract}
We classify all simple transitive $2$-representations for two classes of
finitary $2$-categories associated with tree path algebras and also for
one class of fiat $2$-categories associated with truncated polynomial
rings. Additionally, we compute the Drinfeld centers for all these $2$-categories.

\end{abstract}

\maketitle

\section{Introduction and description of results}

Motivated by the results of \cite{BFK,Kh}, higher representation theory, as the study of
$2$-rep\-re\-sen\-ta\-ti\-ons of additive $2$-categories, originated from the papers \cite{CR,Ro0}.
Further developments in \cite{KL,Ro} motivated development of abstract $2$-representation
theory of finitary $2$-categories in the series \cite{MM1,MM2,MM3,MM4,MM5,MM6} of papers
which formulated and investigated an abstract general setup for the study of natural
$2$-analogues of finite dimensional algebras, called {\em finitary} $2$-categories.
We refer to \cite{Maz} for a general overview.

The ``correct'' $2$-analogue of the notion of an irreducible representation is the notion of
{\em simple transitive} $2$-representation as defined in \cite{MM5}. These $2$-representations
for ``building blocks'' for all $2$-representations, however, the building procedure itself is
more complicated than in the classical representation theory as no good analogue of homological
algebra is available in the $2$-setting for now. Nevertheless, the question of classification
of simple transitive $2$-representations is natural and provides the first layer of information
during the study of a given finitary $2$-category. This question was answered in \cite{MM5,MM6}
for a certain class of finitary (even fiat) $2$-categories, where it was shown that, under
some mild combinatorial assumptions, each simple transitive $2$-representation is equivalent to
a so-called {\em cell $2$-representation}, a class of $2$-representations defined in \cite{MM1}
and further studied in \cite{MM2} and \cite{MM3}. Meanwhile, many new example of finitary
(but not necessarily fiat) $2$-categories were constructed, see for example \cite{GM1,GM,Xa,Zh}
and references therein. In the general case the problem of classification of simple transitive
$2$-representation is wide open. It is not even known whether, for a given finitary $2$-category,
the number of equivalence classes of simple transitive $2$-representations is always finite.

The center of an algebra plays, naturally, a central role in the representation theory of this
algebra. For $2$-categories, an appropriate $2$-analogue of a center is the so-called
{\em Drinfeld center}, as defined in \cite{JS,Maj,MS} in various setups. This is an important
invariant of a $2$-category which is, however, not easy to determine.

The aim of the present article is:
\begin{itemize}
\item to construct new examples of finitary $2$-categories;
\item to classify simple transitive $2$-representations for these new examples and also for
some examples which already appeared in the literature;
\item to describe the Drinfeld center in all these examples.
\end{itemize}

The article is organized as follows: In Section~\ref{s2} we collect all necessary preliminaries
on finitary and fiat $2$-categories. Section~\ref{s3} provides a classification of equivalence
classes of simple transitive $2$-representations for the finitary $2$-category associated to a
tree path algebra, as considered in~\cite{Zh}. The Drinfeld center of this $2$-category is also
described and turns out to be very small.
In Section~\ref{s4}, we define a new finitary $2$-category using functors on certain
quiver algebras, given by tensoring with identity bimodules, the left action of which is
twisted by a class of algebra endomorphisms. We calculates the cell structure of these
$2$-categories in the case when the quiver is a Dynkin diagram of type $A$ with uniform orientation.
The Drinfeld center of this new finitary $2$-category turns out to be quite big. In Section~\ref{s5},
we consider the fiat $2$-category given by twisting functors corresponding to certain
algebra automorphisms of truncated polynomial rings.
In this case we get a fiat $2$-category with unique left cell (resp. right cell) which does not
satisfy the strong regularity assumption from \cite{MM1,MM5}. In particular, the main approach of
\cite{MM5} is not applicable for this $2$-category, however, we reduce the problem of description of
its simple transitive $2$-representations to another result in \cite{MM5}. Nevertheless, we classify all simple
transitive $2$-representations for this $2$-category and describe its Drinfeld center.
It turns out that in this case there are many simple transitive $2$-representations which are
not cell $2$-representations. Moreover, it turns out that the Drinfeld center of this $2$-category is
rather non-trivial.

\textbf{Acknowledgment.}
The paper was written during a visit of the author to Uppsala University, whose hospitality is
gratefully acknowledged. The visit was supported by China Scholarship Council. The author would
like to thank Volodymyr Mazorchuk for useful discussions and valuable comments.
The author is also grateful to the referee for
numerous comments and suggestions which helped to improve the paper greatly.

\section{Preliminaries}\label{s2}

\subsection{Notation}\label{s2.00}

Throughout the paper we will work over a fixed field $\Bbbk$ if not stated explicitly. For simplicity, we assume that it is
algebraically closed.

We let $\mathbf{Cat}$ denote the category of small categories. By a $2$-category we mean a category
enriched over $\mathbf{Cat}$. Thus, a $2$-category consists of
\begin{itemize}
\item objects denoted by $\mathtt{i},\mathtt{j},\ldots$;
\item $1$-morphisms denoted by $F,G,\ldots$;
\item $2$-morphisms denoted by $\alpha,\beta,\ldots$;
\item identity $1$-morphisms $\mathbbm{1}_\mathtt{i}$, for $\mathtt{i}\in\mathscr{C}$;
\item identity $2$-morphisms $\mathrm{id}_F$, for a $1$-morphism $F$;
\item composition $\circ$ of $1$-morphisms;
\item horizontal composition $\circ_0$ of $2$-morphisms;
\item vertical composition  $\circ_1$ of $2$-morphisms.
\end{itemize}
For simplicity, given a $1$-morphism $F$ and a composable $2$-morphism $\alpha$,
we write $F(\alpha)$ for $\mathrm{id}_F\circ_0\alpha$ and $\alpha_F$ for $\alpha\circ_0\mathrm{id}_F$.

\subsection{Finitary and fiat $2$-categories}

An additive $\Bbbk$-linear category is called \textit{finitary} if it is idempotent split, has finitely many
isomorphism classes of indecomposable objects and finite dimensional $\Bbbk$-vector spaces of morphisms.
We denote by $\mathfrak{A}_{\Bbbk}^f$ the $2$-category whose objects are finitary additive $\Bbbk$-linear
categories, $1$-morphisms are additive $\Bbbk$-linear functors and $2$-morphisms are natural
transformations of functors.

A \textit{finitary} $2$-category (over $\Bbbk$) is defined to be a $2$-category $\mathscr{C}$ such that:
\begin{itemize}
\item it has finitely many objects;
\item for each pair $\mathtt{i}$, $\mathtt{j}$ of objects, the category $\mathscr{C}(\mathtt{i},\mathtt{j})$
lies in $\mathfrak{A}_{\Bbbk}^f$ and horizontal composition is biadditive and $\Bbbk$-linear;
\item all identity $1$-morphisms are indecomposable.
\end{itemize}
We refer to \cite{Le,Mc} for generalities on abstract $2$-categories and to
\cite{MM1,MM2,MM3,MM4,MM5,MM6} for more information on finitary $2$-categories.

A finitary $2$-category is called {\em fiat}, see \cite{MM1}, provided that
\begin{itemize}
\item there is a weak involution $\ast:\mathscr{C}\to\mathscr{C}^{\mathrm{op}}$, where
$\mathscr{C}^{\mathrm{op}}$ denote the opposite category in which the directions of both $1$-morphisms
and $2$-morphisms are reversed;
\item for any pair $\mathtt{i}, \mathtt{j}\in\mathscr{C}$ and any $1$-morphism
$F\in\mathscr{C}(\mathtt{i},\mathtt{j})$, there exist $2$-morphisms
$\alpha:F\circ F^\ast\to\mathbbm{1}_\mathtt{j}$ and $\beta:\mathbbm{1}_\mathtt{i}\to F^\ast\circ F$ such that
\begin{displaymath}
\alpha_F\circ_1F(\beta)=\mathrm{id}_F \quad \text{and}\quad F^\ast(\alpha)\circ_1\beta_{F^\ast}=\mathrm{id}_{F^\ast}.
\end{displaymath}
\end{itemize}

\subsection{2-representations}

Let $\mathscr{C}$ be a finitary $2$-category. A $2$-representation of $\mathscr{C}$ is a strict $2$-functor
from $\mathscr{C}$ to $\mathbf{Cat}$. A {\em finitary} $2$-representation of $\mathscr{C}$
is a strict $2$-functor from $\mathscr{C}$ to $\mathfrak{A}_{\Bbbk}^f$. We will usually denote
$2$-representations by $\mathbf{M}, \mathbf{N}, \ldots$. For each $\mathtt{i}\in \mathscr{C}$, we denote by $\mathbf{P}_{\mathtt{i}}$ the i-th {\em principal} $2$-representation $\mathscr{C}(\mathtt{i},{}_-)$.
All finitary $2$-representations of $\mathscr{C}$ form a $2$-category, denoted by $\mathscr{C}$-afmod,
with $2$-natural transformations as $1$-morphisms and modifications as $2$-morphisms, see \cite{Le,MM3}.

Given two $2$-representations $\mathbf{M}$, $\mathbf{N}$ of $\mathscr{C}$, we say they are {\em equivalent}
if there is a $2$-natural transformation $\Phi:\mathbf{M}\to\mathbf{N}$ such that $\Phi_{\mathtt{i}}$
is an equivalence of categories for each $\mathtt{i}$.

Consider a $2$-representation $\mathbf{M}$ of $\mathscr{C}$ and assume that $\mathbf{M}(\mathtt{i})$
is additive and idempotent split for each $\mathtt{i}\in\mathscr{C}$. For any collection of objects
$X_i\in \mathbf{M}(\mathtt{i}_i)$, where $i\in I$, the additive closure of all objects of the form
$\mathbf{M}(F)X_i$, where $i\in I$ and $F$ runs through all $1$-morphisms of $\mathscr{C}$, has the
structure of a $2$-representation of $\mathscr{C}$ by restriction (see \cite{MM5}). We denote this
$2$-subrepresentation of $\mathbf{M}$ by $\mathbf{G}_{\mathbf{M}}(\{X_i:i\in I\})$.
To simplify notation, we will write $FX$ instead of $\mathbf{M}(F)X$ for any $1$-morphism $F$.

Let $\mathbf{M}$ be a finitary $2$-representation of $\mathscr{C}$. For each $1$-morphism
$F\in\mathscr{C}(\mathtt{i},\mathtt{j})$, denote by $[F]$ the matrix with non-negative integer
coefficients whose rows are indexed by isomorphism classes of indecomposable objects in
$\mathbf{M}(\mathtt{j})$, columns are indexed by isomorphism classes of indecomposable objects
in $\mathbf{M}(\mathtt{i})$ and the entry in the position $(Y,X)$ is the multiplicity of $Y$
as a direct summand of $FX$.

\subsection{Combinatorics of finitary $2$-categories}

For a finitary $2$-category $\mathscr{C}$, we denote by $\mathcal{S}_{\mathscr{C}}$ the set of
isomorphism classes of all indecomposable $1$-morphisms in $\mathscr{C}$ with an added external
zero element $0$. From \cite[Section~3]{MM2}, we see that the set $\mathcal{S}_{\mathscr{C}}$
forms a {\em multisemigroup} (for more details, see \cite{KM}), which can be equipped with several
natural preorders. For any two $1$-morphisms $F$ and $G$, we say $G\geq_L F$ in the {\em left preorder}
provided that there is a $1$-morphism $H$ such that $G$ occurs as a direct summand of $H\circ F$,
up to isomorphism. A {\em left cell} is an equivalence class for $\geq_L$. Analogously one defines
the {\em right} and {\em two-sided} preorders $\geq_R$ and $\geq_J$ and the corresponding
{\em right} and {\em two-sided} cells.

\subsection{$2$-ideals}

For a $2$-category $\mathscr{C}$, a {\em left} 2-{\em ideal} $\mathscr{I}$ of $\mathscr{C}$ has
the same objects as $\mathscr{C}$ and for each pair $\mathtt{i}, \mathtt{j}$ of objects we have that
$\mathscr{I}(\mathtt{i}, \mathtt{j})$ is an ideal in $\mathscr{C}(\mathtt{i}, \mathtt{j})$ such that
$\mathscr{I}$ is closed under the left horizontal multiplication with both $1$- and $2$-morphisms
in $\mathscr{C}$. {\em Right } 2-{\em ideals} and {\em two-sided ideals} (which are, simply,
called 2-{\em ideals}) can be defined similarly. For example, principal $2$-representations are left
ideals in $\mathscr{C}$.

Let $\mathbf{M}$ be a $2$-representation of $\mathscr{C}$. An ideal $\mathbf{I}$ in $\mathbf{M}$
is a collection of ideals $\mathbf{I}(\mathtt{i})$ in $\mathbf{M}(\mathtt{i})$ for each
$\mathtt{i}\in \mathscr{C}$ which are stable under the action of $\mathscr{C}$ in the sense that:
for any morphism $\eta\in \mathbf{I}$ and any $1$-morphism $F$ the morphism $\mathbf{M}(F)(\eta)$
is in $\mathbf{I}$ whenever if it is defined.

\subsection{Abelianization}

For a finitary additive $\Bbbk$-linear category $\mathcal{A}$, its {\em abelianization} is the
abelian category $\overline{\mathcal{A}}$ with objects being diagrams of the form $X\overset{\eta}{\to}Y$
for $X,Y\in \mathcal{A}$ and $\eta\in\mathcal{A}(X,Y)$ and morphisms being equivalence classes of
solid commutative diagrams of the form
\begin{displaymath}
\xymatrix{X\ar[rr]^{\eta}\ar[d]^{\tau}&&Y\ar[d]^{\zeta}\ar@{-->}[dll]_{\xi}\\
X'\ar[rr]_{\eta'}&&Y' }
\end{displaymath}
modulo the subspace spanned by diagrams for which there is $\xi$ displayed by the dashed arrow such that $\eta'\xi=\zeta$, see \cite{Fr}.
Let $P$ be a multiplicity-free direct sum of representatives of all isomorphism classes of
indecomposable objects in $\mathcal{A}$. Then, directly from the definitions, we have
$\mathcal{A}\simeq \mathrm{End}_{\mathcal{A}}(P)^{\text{op}}$-proj. Therefore
$\mathcal{A}$ always has weak kernels, c.f.~\cite[Theorem~3.4]{HH}, which implies that $\overline{\mathcal{A}}$
is abelian. We have $\overline{\mathcal{A}}\simeq \mathrm{End}_{\mathcal{A}}(P)^{\text{op}}$-mod,
see~\cite[Propsition~5.3]{AM}.

Given a finitary $2$-category $\mathscr{C}$ and a finitary $2$-representation $\mathbf{M}$ of $\mathscr{C}$,
the abelianization of $\mathbf{M}$ is the $2$-representation $\overline{\mathbf{M}}$ of $\mathscr{C}$ which
sends each $i\in \mathscr{C}$ to the category $\overline{\mathbf{M}}(\mathtt{i})$ and with the action of
$\mathscr{C}$ defined on diagrams component-wise.

\subsection{Cell $2$-representations}

Let $\mathscr{C}$ be a finitary $2$-category and $\mathcal{L}$ a left cell. Since multiplication from
the left does not change the source of the original morphism, there is
an $\mathtt{i}=\mathtt{i}_\mathcal{L}\in\mathscr{C}$ such that for any $1$-morphism
$F\in\mathcal{L}$ we have $F\in\mathscr{C}(\mathtt{i},\mathtt{j})$ for some $\mathtt{j}\in\mathscr{C}$.
For $\mathtt{j}\in\mathscr{C}$ denote by $\mathbf{N}(\mathtt{j})$ the additive closure in
$\mathbf{P}_\mathtt{i}(\mathtt{j})$ of all $1$-morphisms $F\in\mathscr{C}(\mathtt{i},\mathtt{j})$ such
that $F\geq_L \mathcal{L}$, that is, the full subcategory of
$\mathbf{P}_\mathtt{i}(\mathtt{j})$ consisting of all objects
which are isomorphic to direct summands of finite direct sums of all such $1$-morphisms $F$.
Then $\mathbf{N}$ is a $2$-subrepresentation of $\mathbf{P}_\mathtt{i}$.
By \cite[Lemma~3]{MM5}, there exists a unique maximal ideal $\mathbf{I}$ in $\mathbf{N}$ such that
it does not contain $\mathrm{id}_F$ for any $F\in\mathcal{L}$. The corresponding quotient $2$-functor $\mathbf{C}_\mathcal{L}:=\mathbf{N}/\mathbf{I}$ is called the ({\em additive}) {\em cell}
$2$-representations of $\mathscr{C}$ corresponding to  $\mathcal{L}$. The abelianization
$\overline{\mathbf{C}}_\mathcal{L}$ of $\mathbf{C}_\mathcal{L}$ is called the {\em abelian cell}
$2$-representation of $\mathscr{C}$ corresponding to  $\mathcal{L}$.

\subsection{Simple transitive $2$-representations}

Let $\mathscr{C}$ be a finitary $2$-category. A finitary $2$-representation $\mathbf{M}$ of $\mathscr{C}$
is called {\em transitive} provided that for every $\mathtt{i}$ and every non-zero object
$X\in\mathbf{M}(\mathtt{i})$ we have $\mathbf{G}_\mathbf{M}(\{X\})=\mathbf{M}$. By \cite[Lemma~4]{MM5},
each transitive $2$-representation $\mathbf{M}$ contains a unique maximal ideal $\mathbf{I}$ which does
not contain any identity morphisms apart from the one for the zero object.  Denote by $\underline{\mathbf{M}}$
the quotient of $\mathbf{M}$ by this ideal $\mathbf{I}$.

A transitive $2$-representation $\mathbf{M}$ of $\mathscr{C}$ is called {\em simple transitive} provided
that $\mathbf{I}$ is the zero ideal or, alternatively, $\mathbf{M}=\underline{\mathbf{M}}$. For
any transitive $2$-representation $\mathbf{M}$ of $\mathscr{C}$, the $2$-representation
$\underline{\mathbf{M}}$ is simple transitive and is called the {\em simple transitive quotient} of $\mathbf{M}$.

\subsection{Drinfeld center for bicategories}\label{s2.9}

By a $2$-category we always mean a {\em strict} $2$-category and the term {\em bicategory} is used
for the corresponding non-strict structure, see \cite{Le}. Note that any bicategory is biequivalent
to a $2$-category, see \cite[Section~2.3]{Le}.

The notion of Drinfeld center originates in \cite[Example~3.4]{JS} and \cite[Definition~3]{Maj}
where it was given for tensor categories, that is bicategories with one object. In \cite{MS},
E.~Meir and M.~Szymik extended the notion of Drinfeld center to cover any bicategory. As all $2$-categories
considered in this paper only have one object, it is convenient to give the original definition
for the case when $\mathscr{B}$ only has one object $\mathtt{i}$.

In the latter case, the {\em Drinfeld center} $\mathcal{Z}(\mathscr{B})$ is a category, whose objects are
pairs $(F,\Phi)$, where $F\in \mathscr{B}(\mathtt{i},\mathtt{i})$ and $\Phi$ is a natural isomorphism
from the functor $F\circ{}_-$ to the functor ${}_-\circ F$ such that
\begin{equation}\label{1w}
\Phi(\mathbbm{1}_i):F\circ \mathbbm{1}_{\mathtt{i}}\cong F\cong\mathbbm{1}_{\mathtt{i}}\circ F
\end{equation}
and
\begin{equation}\label{1c}
\Phi(K\circ H)=(\mathrm{id}_{K}\circ_0 \Phi(H) )\circ_1(\Phi(K)\circ_0 \mathrm{id}_{H})
\end{equation}
holds for any $1$-morphisms $K,H\in \mathscr{B}(\mathtt{i},\mathtt{i})$ whenever the composition
$K\circ H$ makes sense. The morphisms between any two objects $(F,\Phi)$ and $(G,\Psi)$ are given
by all morphisms $f$ in $\mathrm{Hom}_{\mathscr{B}(\mathtt{i},\mathtt{i})}(F,G)$ such that
\begin{equation}\label{2c}
(\mathrm{id}_K\circ_0 f)\circ_1\Phi(K)=\Psi(K)\circ_1(f\circ_0 \mathrm{id}_K)
\end{equation}
for all $1$-morphisms $K\in\mathscr{B}(\mathtt{i},\mathtt{i})$. The category $\mathcal{Z}(\mathscr{B})$
has the natural structure of a tensor category via
\begin{equation}\label{3c}
(F,\Phi)\circ (G,\Psi)=(F\circ G, (\Phi\circ_0 \mathrm{id})\circ_1(\mathrm{id}\circ_0 \Psi)),
\end{equation}
with the tensor unit $(\mathbbm{1}_{\mathtt{i}},e)$, where
$e(F):\mathbbm{1}_{\mathtt{i}}\circ F\cong F\cong F\circ \mathbbm{1}_{\mathtt{i}}$ for any $1$-morphism
$F\in\mathscr{B}(\mathtt{i},\mathtt{i})$. If $\mathscr{B}$ is a $2$-category, then each $e(F)$
is exactly the identity morphism $\mathrm{id}_F$.

\begin{remark}\label{re1}{\rm
If $\mathscr{B}$ is a $2$-category, then we have the following:
\begin{enumerate}[$($i$)$]
\item \label{re12} Condition \eqref{1w} turns to $\Phi(\mathbbm{1}_i)=\mathrm{id}_F$,
which is redundant since it can be deduced from \eqref{1c} and
the fact that $\Phi(\mathbbm{1}_i)$ is an isomorphism, see \cite[Lemma~3.2]{Mug}.
\item The product $K^n:=\underbrace{K\circ K\circ \cdots \circ K}_{n\text{ factors}}$
is well-defined, for any $1$-morphism $K\in \mathscr{B}(\mathtt{i},\mathtt{i})$ and any positive integer $n$,
moreover,  from \eqref{1c} it follows that
\begin{multline}\label{e5}
\Phi(K\circ K\circ \cdots \circ K)=\\
(\mathrm{id}_{K}\circ _0 \mathrm{id}_{K}\circ_0  \cdots \circ_0\mathrm{id}_{K}\circ_0  \Phi(K) )
\circ_1(\mathrm{id}_{K}\circ _0 \mathrm{id}_{K}\circ_0  \cdots \circ_0\Phi(K)\circ_0  \mathrm{id}_{K} )
\circ_1 \cdots\\\cdots\circ_1
(\mathrm{id}_{K}\circ_0  \Phi(K)\circ_0 \mathrm{id}_{K}\circ_0\cdots \circ_0  \mathrm{id}_{K})\circ_1
(\Phi(K)\circ_0  \mathrm{id}_{K}\circ_0 \cdots\circ_0  \mathrm{id}_{K}),
\end{multline}
where the product in each bracket has $n$ factors.
\item \label{re22}
If $\mathscr{B}(\mathtt{i},\mathtt{i})$ has direct sums,
then so does $\mathcal{Z}(\mathscr{B})$, see \cite[Lemma~3.6]{Mug}.
In more detail, for any two pairs $(F_1,\Phi^{(1)}), (F_2,\Phi^{(2)})\in \mathcal{Z}(\mathscr{B})$,
their direct sum~$(F_1\oplus F_2,\Theta)$ is defined as follows:
for any $1$-morphism $K\in \mathscr{B}(\mathtt{i},\mathtt{i})$,
the isomorphism $\Theta(K)$ is given by the sum of the two paths of maximal length on the following diagram:
\begin{equation}\label{eq321}
\xymatrix@C=2.3pc@R=2.3pc{&(F_1\oplus F_2)\circ K\ar[dl]_{\pi_{F_1}\circ_0 \mathrm{id}_K}
\ar[dr]^{\pi_{F_2}\circ_0 \mathrm{id}_K}&\\
F_1\circ K\ar[d]_{\Phi^{(1)}(K)}&&F_2\circ K\ar[d]^{\Phi^{(2)}(K)}\\
K\circ F_1\ar[dr]_{ \mathrm{id}_K\circ_0\iota_{F_1}}&& K\circ F_2
\ar[dl]^{ \mathrm{id}_K\circ_0\iota_{F_2}}\\
&K\circ(F_1\oplus F_2)&
}
\end{equation}
The inverse of  $\Theta(K)$ is given by the sum of the two paths of maximal length on the following diagram:
\begin{displaymath}
\xymatrix@C=2.3pc@R=2.3pc{&K\circ(F_1\oplus F_2)\ar[dl]_{\mathrm{id}_K\circ_0 \pi_{F_1}}
\ar[dr]^{\mathrm{id}_K\circ_0 \pi_{F_2}}&\\
K\circ F_1\ar[d]_{(\Phi^{(1)}(K))^{-1}}&&K\circ F_2\ar[d]^{(\Phi^{(2)}(K))^{-1}}\\
F_1\circ K\ar[dr]_{ \iota_{F_1}\circ_0\mathrm{id}_K}&& F_2\circ K
\ar[dl]^{ \iota_{F_2}\circ_0\mathrm{id}_K}\\
&(F_1\oplus F_2)\circ K&
}
\end{displaymath}
It follows directly from the definition that $\Theta$ is a natural isomorphism.
For any $1$-morphisms $K,H\in \mathscr{B}(\mathtt{i},\mathtt{i})$, condition~\eqref{1c} for $\Theta$
follows from the commutativity of the following diagram:
\begin{displaymath}
\xymatrix@C=2.3pc@R=2.3pc{&(F_1\oplus F_2)\circ K\circ H\ar[dl]_{\pi_{F_1}\circ_0 \mathrm{id}_K\circ_0\mathrm{id}_H}
\ar[dr]^{\pi_{F_2}\circ_0 \mathrm{id}_K\circ_0\mathrm{id}_H}&\\
F_1\circ K\circ H\ar@/_4pc/[dddd]_{\Phi^{(1)}(K\circ H)}\ar[d]^{\Phi^{(1)}(K)\circ_0\mathrm{id}_H}&&F_2\circ K\circ H\ar[d]_{\Phi^{(2)}(K)\circ_0\mathrm{id}_H}\ar@/^4pc/[dddd]^{\Phi^{(2)}(K\circ H)}\\
K\circ F_1\circ H\ar@{=}[dd]\ar[dr]^{ \mathrm{id}_K\circ_0\iota_{F_1}\circ_0\mathrm{id}_H}&& K\circ F_2 \circ H\ar@{=}[dd]
\ar[dl]_{ \mathrm{id}_K\circ_0\iota_{F_2}\circ_0\mathrm{id}_H}\\
&K\circ(F_1\oplus F_2)\circ H\ar[dl]^{\mathrm{id}_K\circ_0\pi_{F_1}\circ_0\mathrm{id}_H}
\ar[dr]_{\mathrm{id}_K\circ_0\pi_{F_2}\circ_0\mathrm{id}_H}&\\
K\circ F_1\circ H\ar[d]^{\mathrm{id}_K\circ_0\Phi^{(1)}(H)}&&K\circ F_2\circ H\ar[d]_{\mathrm{id}_K\circ_0\Phi^{(2)}(H)}\\
K\circ H\circ F_1\ar[dr]_{ \mathrm{id}_K\circ_0\mathrm{id}_H\circ_0\iota_{F_1}}&& K\circ H \circ F_2
\ar[dl]^{ \mathrm{id}_K\circ_0\mathrm{id}_H\circ_0\iota_{F_2}}\\
&K\circ H\circ(F_1\oplus F_2)&
}
\end{displaymath}
Here the top (resp. bottom) hexagon in the middle is obtained by horizontally composing
the diagram of \eqref{eq321}
with $\mathrm{id}_H$ (resp. $\mathrm{id}_K$) from the right (resp. left) hand side.
The two triangles in the middle obviously commute and the two ``rectangles'' on the sides
commute by condition~\eqref{1c}.  By definition, the right hand side of condition~\eqref{1c} for 
$\Theta$ is the sum of four paths of maximal length obtained from the two hexagons. 
Two of these paths are zero because of $\pi_{F_2}\circ_1\iota_{F_1}=0$ and $\pi_{F_1}\circ_1\iota_{F_2}=0$.
Hence condition~\eqref{1c} for $\Theta$ follows.
\end{enumerate}}
\end{remark}

Conversely to Remark~\ref{re1}~\eqref{re22}, we have the following statement:

\begin{proposition}\label{prop}
For any object $(F_1\oplus F_2,\Theta)\in \mathcal{Z}(\mathscr{B})$, we have that $(F_1\oplus F_2,\Theta)$ can
be decomposed into a direct sum of two objects $(F_1,\Phi^{(1)}),(F_2,\Phi^{(2)})\in\mathcal{Z}(\mathscr{B})$ 
if, and only if, 
for any $1$-morphism $K\in \mathscr{B}(\mathtt{i},\mathtt{i})$ and $i\neq j\in\{1,2\}$, we have
\begin{equation}\label{eq1}
(\mathrm{id}_{K} \circ_0 \pi_{F_j})
\circ_1\Theta(K)\circ_1(\iota_{F_i}\circ_0 \mathrm{id}_{K})=0,
\end{equation}
or, equivalently, for any $1$-morphism $K\in \mathscr{B}(\mathtt{i},\mathtt{i})$ and $i\neq j\in\{1,2\}$, we have
\begin{equation}\label{eq2}
(\pi_{F_j}\circ_0 \mathrm{id}_K)\circ_1
(\Theta(K))^{-1}\circ_1 (\mathrm{id}_K\circ_0 \iota_{F_i})=0.
\end{equation}
\end{proposition}

\begin{proof}
It directly follows from the definition and Remark~\ref{re1}\eqref{re22} that any direct sum of two objects in 
$\mathcal{Z}(\mathscr{B})$ satisfies both conditions~\eqref{eq1} and~\eqref{eq2}.

We now prove the ``only if" part of the statement. It is clear that condition~\eqref{eq1} 
and  condition~\eqref{eq2} are equivalent, so we assume that condition~\eqref{eq1} is satisfied.
For any $1$-morphism $K\in \mathscr{B}(\mathtt{i},\mathtt{i})$, we define $\Phi^{(i)}(K)$, $i=1,2$,
by requiring that the solid square in the following diagram commutes:
\begin{displaymath}
\xymatrix@R=3.3pc@C=2.5pc{F_i\circ K\ar@<+.5ex>[rr]^{\iota_{F_i}\circ_0 \mathrm{id}_K\quad}\ar@<-.5ex>[d]_{\Phi^{(i)}(K)}
&&(F_1\oplus F_2)\circ K\ar@<+.5ex>@{-->}[ll]^{\pi_{F_i}\circ_0 \mathrm{id}_K\quad}
\ar@<-.5ex>[d]_{\Theta(K)}\\
K\circ F_i\ar@<-.5ex>@{-->}[u]_{(\Phi^{(i)}(K))^{-1}}\ar@<-.5ex>@{-->}[rr]_{\mathrm{id}_K\circ_0\iota_{F_i}\quad}&
&K\circ (F_1\oplus F_2)\ar@<-.5ex>@{-->}[u]_{(\Theta(K))^{-1}}
\ar@<-.5ex>[ll]_{\mathrm{id}_{K}\circ_0 \pi_{F_i}\quad}.}
\end{displaymath}
Similarly, we define $(\Phi^{(i)}(K))^{-1}$ so that the dashed square in the above diagram commutes.
From the definition, we have that, for each $i$, the corresponding  $\Phi^{(i)}(K)$'s give a natural 
transformation $\Phi^{(i)}$.
Indeed, for any $i,j$ and any $2$-morphism $\alpha: K\to H$, we have the following commutative diagram 
\begin{equation}\label{eq01}
\xymatrix@R=3.3pc@C=2.3pc{F_i\circ K\ar[rr]^{\iota_{F_i}\circ_0 \mathrm{id}_K\quad}\ar[d]_{\mathrm{id}_{F_i}\circ_0\alpha}
&&(F_1\oplus F_2)\circ K
\ar[rr]^{\Theta(K)}\ar[d]_{\mathrm{id}_{F_1\oplus F_2}\circ_0\alpha}
&& K\circ (F_1\oplus F_2)
\ar[d]_{\alpha\circ_0\,
\mathrm{id}_{F_1\oplus F_2}}\ar[rr]^{\quad\mathrm{id}_{K} \circ_0 \pi_{F_j}}
&& K\circ F_j
\ar[d]_{\alpha\circ_0\,
\mathrm{id}_{F_j}}\\
F_i\circ H\ar[rr]^{\iota_{F_i}\circ_0 \mathrm{id}_{H}\quad}&&(F_1\oplus F_2)\circ H
\ar[rr]^{\Theta(H)}&&
 H\circ (F_1\oplus F_2)\ar[rr]^{\quad\mathrm{id}_{H}\circ_0\pi_{F_j}}&&H\circ F_j.}
\end{equation}
Note that $\sum_{i} \iota_{F_i}\circ_1\pi_{F_i}=\mathrm{id}_{F_1\oplus F_2}$.
By condition~\eqref{eq1} and definitions, for any $1$-morphism $K$, we have  
\begin{displaymath}
\Phi^{(i)}(K)\circ_1(\Phi^{(i)}(K))^{-1}=\mathrm{id}_{K\circ F_i}\quad \text{and}\quad
(\Phi^{(i)}(K))^{-1}\circ_1\Phi^{(i)}(K)= id_{F_i\circ K}, 
\end{displaymath}
which implies that each $\Phi^{(i)}$ is, indeed, a natural isomorphism.
It is also easy to  check condition~\eqref{1c}, for each $\Phi^{(i)}$. 
This 
is given by the two ``rectangles'' on the sides of the following diagram
(commutativity of this diagram uses condition~\eqref{eq1}):
\begin{displaymath}
\xymatrix@C=2.3pc@R=2.3pc{&(F_1\oplus F_2)\circ K\circ H\ar[ddd]^{\Theta(K)\circ_0 \mathrm{id}_H}&\\
F_1\circ K\circ H\ar[ur]^{\iota_{F_1}\circ_0 \mathrm{id}_K\circ_0\mathrm{id}_H}\ar@/_4pc/[dddd]_{\Phi^{(1)}(K\circ H)}
\ar[d]^{\Phi^{(1)}(K)\circ_0\mathrm{id}_H}&&F_2\circ K\circ 
H\ar[d]_{\Phi^{(2)}(K)\circ_0\mathrm{id}_H}\ar@/^4pc/[dddd]^{\Phi^{(2)}(K\circ H)}
\ar[ul]_{\iota_{F_2}\circ_0 \mathrm{id}_K\circ_0\mathrm{id}_H}\\
K\circ F_1\circ H\ar@{=}[dd]&& K\circ F_2 \circ H\ar@{=}[dd]\\
&K\circ(F_1\oplus F_2)\circ H \ar[ul]_{\mathrm{id}_K\circ_0\pi_{F_1}\circ_0\mathrm{id}_H}
\ar[ur]^{ \mathrm{id}_K\circ_0\pi_{F_2}\circ_0\mathrm{id}_H}
\ar[ddd]^{\mathrm{id}_K\circ_0\Theta(H)}&\\
K\circ F_1\circ H\ar[d]^{\mathrm{id}_K\circ_0
\Phi^{(1)}(H)}\ar[ur]_{\mathrm{id}_K\circ_0\iota_{F_1}\circ_0\mathrm{id}_H}&
&K\circ F_2\circ H\ar[d]_{\mathrm{id}_K\circ_0\Phi^{(2)}(H)}\ar[ul]^{\mathrm{id}_K\circ_0\iota_{F_2}\circ_0\mathrm{id}_H}\\
K\circ H\circ F_1&& K\circ H \circ F_2\\
&K\circ H\circ(F_1\oplus F_2)\ar[ur]_{ \mathrm{id}_K\circ_0\mathrm{id}_H\circ_0\pi_{F_2}}\ar[ul]^{ \mathrm{id}_K\circ_0\mathrm{id}_H\circ_0\pi_{F_1}}&
}
\end{displaymath}
\end{proof}
Similarly to~\eqref{eq01}, for any $i,j$ and 
any $2$-morphism $\alpha: K\to H$, we have the following commutative diagram
\begin{equation}\label{eq3}
\xymatrix@R=3.3pc@C=2.3pc{ K \circ F_i\ar[rr]^{\mathrm{id}_K\circ_0 \iota_{F_i}\quad}
\ar[d]_{\alpha\circ_0\mathrm{id}_{F_i}}
&&K \circ(F_1\oplus F_2)
\ar[rr]^{(\Theta(K))^{-1}}\ar[d]_{\alpha\circ_0\mathrm{id}_{F_1\oplus F_2}}
&& (F_1\oplus F_2)\circ K
\ar[d]_{\mathrm{id}_{F_1\oplus F_2}\circ_0\alpha}\ar[rr]^{\quad \pi_{F_j}\circ_0\mathrm{id}_{K}}
&&  F_j\circ K
\ar[d]_{\mathrm{id}_{F_j}\circ_0\alpha}\\
H \circ F_i\ar[rr]^{\mathrm{id}_{H}\circ_0 \iota_{F_i}\quad}&&H \circ(F_1\oplus F_2)
\ar[rr]^{(\Theta(H))^{-1}}&&
(F_1\oplus F_2)\circ H\ar[rr]^{\quad\pi_{F_j}\circ_0\mathrm{id}_{H}}&& F_j\circ H.}
\end{equation}

\begin{remark}{\rm
Later on, in Theorem~\ref{th28'}~\eqref{th28'.1} (for the case $k=d=2$), we will see examples of  
indecomposable elements in the Drinfeld center which, in particular, have the form $(F_0\oplus F_0,\Phi)$.
}
\end{remark}

\section{Finitary $2$-category of ideals for a tree algebra}\label{s3}

\subsection{A finitary $2$-category for tree algebra}\label{s3.1}

Let $A$ be the path algebra of a finite connected tree quiver $Q=(Q_0, Q_1, \mathfrak{s}, \mathfrak{t})$, where $Q_0$
is the set of vertices, $Q_1$ is the set of arrows, $\mathfrak{s}:Q_1\to Q_0$ is the source function and
$\mathfrak{t}:Q_1\to Q_0$ is the target function. Denote by $Q^p$ the set consisting of all paths in $Q$ and
by $\mathfrak{l}:Q^p\to \{0,1,2,\dots\}$ the {\em length function}, which assigns the length of the path
to each path. The set $Q^p$ can be equipped with a partial order given, for  $w,w'\in Q^p$, by
$w\preceq w'$ if $w'=awb$ for some $a,b\in Q^p$. We also write $w\prec w'$ if $w\preceq w'$ and $w\neq w'$.
For each vertex $v\in Q_0$, we denote by $\varepsilon_v$ the corresponding trivial path in $Q$ of length zero
and in this way we identify vertices in $Q$ with paths of length zero.

Let $\mathcal{I}(A)$ denote the set consisting of all ideals in $A$. Elements in $\mathcal{I}(A)$ can be
alternatively viewed as subbimodules of the $A\text{-}A$-bimodule ${}_AA_A$. We denote by
$\mathcal{I}(A)^{\mathrm{ind}}$ the subset of $\mathcal{I}(A)$ consisting of all indecomposable ideals,
namely, indecomposable subbimodules. By \cite[Lemma~3]{Zh}, each ideal $I$ in $A$ has a unique minimal set
of path generators denoted by $G(I)$. We denote by $\mathfrak{s}_{G(I)}$ the set of all sources for
elements in $G(I)$ and by $\mathfrak{t}_{G(I)}$ the set of all targets for
elements in $G(I)$ respectively.

For each ideal $I$ of $A$, define $\mathrm{Dp}_{I}$ to be the functor
\begin{displaymath}
I\otimes_A {}_-:A\text{-mod}\to A\text{-mod}.
\end{displaymath}

Let $\mathcal{C}$ be a small category equivalent to $A$-mod. Then we define the $2$-category $\mathscr{D}_A$ to have
\begin{itemize}
\item one object $\mathtt{i}$ (which we identify with $\mathcal{C}$);
\item as $1$-morphisms, all functors are given, up to equivalence with $A$-mod, by
functors from the additive closure of all $\mathrm{Dp}_{I}$'s;
\item as $2$-morphisms, all  natural transformations of functors.
\end{itemize}
The category $\mathscr{D}_A$ is a finitary $2$-category but not a fiat one unless $Q$ has only one vertex. Note that $\mathrm{Dp}_{I}\circ \mathrm{Dp}_{J}\cong\mathrm{Dp}_{{IJ}}$ for any ideals $I,J$ of $A$, see \cite{GM}.
It follows from the Krull-Schmidt theorem that every indecomposable $1$-morphism is uniquely (up to isomorphism) determined by an indecomposable ideal.

\subsection{Simple transitive $2$-representations for $\mathscr{D}_A$}

By \cite[Lemma~6]{Zh}, for each ideal $I\in\mathcal{I}(A)^{\mathrm{ind}}$, the corresponding
isomorphism class of the functor $\mathrm{Dp}_{I}$ forms a left cell which we will denote by $\mathcal{L}_I$.
The same isomorphism class forms, as well, a right cell and hence also a two-sided cell.
By definition, we have
\begin{align*}
\mathbf{N}_I(\mathtt{i})=\mathrm{add}(\{F:&\ F \text{ is isomorphic to
a direct summand of}\\
&\text{ an element in }
\{G\circ \mathrm{Dp}_{I}|\ G\in \mathcal{S}_{\mathscr{D}_A}\}\}).
\end{align*}
From \cite[Corollary~8]{Zh}, we obtain that the unique maximal ideal $\mathbf{I}_I$ in $\mathbf{N}_I$
which does not contain the identity $2$-morphism on $\mathrm{Dp}_{I}$
is generated by all $2$-morphisms $\mathrm{id}_F$, where $F>_{\mathcal{L}}\mathrm{Dp}_{I}$. Moreover,
the endomorphism algebra of the object $\mathrm{Dp}_{I}$ in the quotient category $\mathbf{N}_I/\mathbf{I}_I(\mathtt{i})=\mathbf{C}_{\mathcal{L}_I}(\mathtt{i})$ is isomorphic to
$\Bbbk$. Therefore $\mathbf{C}_{\mathcal{L}_I}(\mathtt{i})\cong \Bbbk$-mod.

For a fixed ideal $I\in \mathcal{I}(A)^{\mathrm{ind}}$, set $K=\langle \varepsilon_{i} |\ i\in \mathfrak{t}_{G(I)}\rangle$. It follows from the definition that $K$ is an idempotent ideal and $KI=I$.

\begin{proposition}\label{p1}
For any ideal $I\in \mathcal{I}(A)^{\mathrm{ind}}$ and the corresponding ideal $K$ defined above,
the cell $2$-representations $\mathbf{C}_{\mathcal{L}_I}$ and $\mathbf{C}_{\mathcal{L}_K}$ are equivalent.
\end{proposition}

To prove this proposition, we need the following lemma. We define $\mathrm{St}_I$ as the set
consisting of all ideals $J$ such that $JI=I$.

\begin{lemma}\label{l1}
Given $I$ and $K$ as above, we have $\mathrm{St}_I=\mathrm{St}_K$.
\end{lemma}

\begin{proof}
Since $KI=I$, we obtain the inclusion $\mathrm{St}_K\subset\mathrm{St}_I$.
To prove $\mathrm{St}_I\subset\mathrm{St}_K$,
we assume that $G(I)=\{u_1, u_2,\dots,u_k\}$ and
consider the set $G(J)=\{w_1,w_2,\dots,w_l\}$ for some ideal $J\in \mathrm{St}_I$.
Note that $JI=I$ and the ideal $JI$ is generated by the set
\begin{displaymath}
\{w_iau_s| 1\leq i\leq l, 1\leq s\leq k, a\in Q^p\}.
\end{displaymath}
Then, for each $u_j\in G(I)=G(JI)$, there exist some $i,s$ and $a$ such that $w_iau_s=u_j$.
Since $u_1,\dots, u_k$ form an anti-chain with respect to $\preceq$, we get $s=j$ and
$w_ia=\varepsilon_{\mathfrak{t}(u_j)}$. Furthermore, we have $w_i=\varepsilon_{\mathfrak{t}(u_j)}$
and thus $G(K)\subset G(J)$, which implies $K\subset J$. Thus $JK=K$ and $J\in\mathrm{St}_K$.
This completes the proof.
\end{proof}

\begin{proof}[Proof of Proposition~\ref{p1}]
Consider the endofunctor
${}_-\circ \mathrm{Dp}_{I}:\mathbf{P}_{\mathtt{i}}(\mathtt{i})\to\mathbf{P}_{\mathtt{i}}(\mathtt{i})$.
It sends objects in $\mathbf{N}_K(\mathtt{i})$ to objects in $\mathbf{N}_I(\mathtt{i})$.
Let $F$ be an indecomposable $1$-morphism such that $F>_{\mathcal{L}}\mathrm{Dp}_{K}$
and let $J$ be the ideal defining $F$. Then $J\subsetneq K$ and hence $JK\neq K$.
By Lemma~\ref{l1} we thus get $JI\neq I$, that is $F\circ \mathrm{Dp}_{I}\not\cong \mathrm{Dp}_{I}$.
Taking the first paragraph of this subsection into account,
the latter implies that ${}_-\circ \mathrm{Dp}_{I}$ maps $\mathbf{I}_K(\mathtt{i})$
to $\mathbf{I}_I(\mathtt{i})$ and hence induces a functor from $\mathbf{C}_{\mathcal{L}_K}(\mathtt{i})$ to $\mathbf{C}_{\mathcal{L}_I}(\mathtt{i})$. Note that both $\mathbf{C}_{\mathcal{L}_K}(\mathtt{i})$
and $\mathbf{C}_{\mathcal{L}_I}(\mathtt{i})$ are equivalent to $\Bbbk\text{-}\mathrm{mod}$, moreover,
$KI=I$ implies that ${}_-\circ \mathrm{Dp}_{I}$ sends an indecomposable generator of
$\mathbf{C}_{\mathcal{L}_K}(\mathtt{i})$ to an indecomposable generator of $\mathbf{C}_{\mathcal{L}_I}(\mathtt{i})$.
Therefore ${}_-\circ \mathrm{Dp}_{I}$ defines an equivalence between
$\mathbf{C}_{\mathcal{L}_K}$ to $\mathbf{C}_{\mathcal{L}_I}$.
\end{proof}

\begin{lemma}\label{l1p}
Any idempotent ideal $I\in \mathcal{I}(A)^{\mathrm{ind}}$ is generated by length zero paths, moreover, $\mathrm{St}_I=\{J\in \mathcal{I}(A)|\,I\subset J\}$.
\end{lemma}

\begin{proof}
Consider the ideal $K$ corresponding to $I$ as defined above.
If $I^2=I$, then $I\in \mathrm{St}_I$. By Lemma~\ref{l1}, we get $I\in \mathrm{St}_K$ and thus $K\subset I$.
As $I\subset K$ by construction, we obtain $I=K$ which means that $I$ is generated by length zero paths.

Denote by $\Gamma$ the set $\{J\in \mathcal{I}(A)|\,I\subset J\}$. Clearly, $\mathrm{St}_I\subset \Gamma$.
For any $J\in\Gamma$, we have both
\begin{displaymath}
 I=I^2\subset JI\subset I\quad\text{and}\quad I=I^2\subset IJ\subset I.
 \end{displaymath}
Thus we get $JI=I=IJ$ which implies the  inclusion $\Gamma\subset \mathrm{St}_I$.
\end{proof}

Due to Lemma~\ref{l1p}, for any two distinct idempotent ideals $I,J\in \mathcal{I}(A)^{\mathrm{ind}}$, we have that
either $I\not\in \mathrm{St}_J$ or $J\not\in \mathrm{St}_I$ and thus
the cell $2$-representations $\mathbf{C}_{\mathcal{L}_I}$
and $\mathbf{C}_{\mathcal{L}_J}$ are not equivalent.

Now we are ready to formulate our first main result.

\begin{theorem}\label{th5}
Every simple transitive $2$-representation of $\mathscr{D}_A$ is equivalent to $\mathbf{C}_{\mathcal{L}_I}$
for some indecomposable idempotent ideal $I$.
\end{theorem}

\begin{proof}
Let $\mathbf{M}$ be a simple transitive $2$-representation of $\mathscr{D}_A$. Set
\begin{displaymath}
\Sigma:=\{F\in \mathcal{S}_{\mathscr{D}_A}|\,\mathbf{M}(F)\neq 0\}.
\end{displaymath}
Since $\mathbf{M}(\mathbbm{1}_{\mathtt{i}})=\mathrm{id}_{\mathbf{M}(\mathtt{i})}\neq 0$, we see that the
set $\Sigma$ is not empty. Let $I$ be a minimal (with respect to inclusions) indecomposable ideal of
$A$ such that $\mathrm{Dp}_{I}\in \Sigma$. Then the additive closure of $\mathrm{Dp}_{I}X$,
where $X$ runs through all objects in $\mathbf{M}(\mathtt{i})$, is non-zero
since $\mathrm{Dp}_{I}\in \Sigma$, and is closed under the action of
$\mathscr{D}_A$ by minimality of $I$ and the fact that $I$ is an ideal.
Transitivity of $\mathbf{M}$ hence implies that this additive
closure must coincide with the whole of $\mathbf{M}(\mathtt{i})$. As, for any ideal $J\in\mathcal{I}(A)$,
we have $JI\subset I$, from the minimality of $I$ it follows that $\mathrm{Dp}_{J}$ acts as zero on
$\mathbf{M}(\mathtt{i})$ if and only if $JI\neq I$. In particular, if $JI\neq I$, then none of the
direct summands of $\mathrm{Dp}_{JI}$ lies in $\Sigma$.

Assume that $X_1, X_2, \dots, X_n$ is a complete and irredundant list of representatives of isomorphism
classes of indecomposable objects in $\mathbf{M}(\mathtt{i})$. Since $\mathrm{Dp}_{I}\in \Sigma$,
there exists some $j$ such that $\mathrm{Dp}_{I}X_j\neq 0$. Note that $0\neq\mathrm{add}(\mathrm{Dp}_{I}X_j)$
is $\mathscr{D}_A$-invariant. Due to transitivity of $\mathbf{M}$, we obtain
$\mathrm{add}(\mathrm{Dp}_{I}X_j)=\mathbf{M}(\mathtt{i})$. Therefore we have
$\mathrm{add}(\mathrm{Dp}_{I}^2\mathbf{M}(\mathtt{i}))=\mathbf{M}(\mathtt{i})$
yielding $I^2=I$ by minimality of $I$. By Lemma~\ref{l1p}, this idempotent ideal $I$
is generated by length zero paths.

Now we claim that there exists exactly one minimal indecomposable ideal $I$ of $A$ such that
$\mathrm{Dp}_{I}\in \Sigma$. Indeed, if $I'$ would be another such minimal ideal, then
minimality of both $I$ and $I'$ would imply $I'I\neq I$ which, by the above, would mean that
$\mathrm{Dp}_{I'}\not\in \Sigma$, a contradiction. Therefore, for any $\mathrm{Dp}_{J}\in \Sigma$,
we have $I\subset J$ and hence $JI=IJ=I$.

Next we claim that $\mathrm{Dp}_{I}X_i\neq 0$ for all $i$. Indeed, assume $\mathrm{Dp}_{I}X_{i}=0$
for some $i$. Then, for any  $J$ such that $\mathrm{Dp}_{J}\in \Sigma$, we have
\begin{displaymath}
0=\mathrm{Dp}_{I}X_{i}=\mathrm{Dp}_{IJ}X_{i}\cong\mathrm{Dp}_{I}\mathrm{Dp}_{J}X_{i}.
\end{displaymath}
This means that $\mathbf{G}_\mathbf{M}(\{X_{i}\})$ is annihilated by $\mathrm{Dp}_{I}$
and hence cannot coincide with $\mathbf{M}(\mathtt{i})$ since $\mathrm{Dp}_{I}\in \Sigma$.
This, however, contradicts transitivity of $\mathbf{M}$. Therefore $\mathrm{Dp}_{I}X_i\neq 0$ for all $i$,
moreover, $\mathrm{add}(\mathrm{Dp}_{I}X_i)$ is $\mathscr{D}_A$-invariant for each $i$, since $I$ is an ideal,
and thus must coincide with $\mathbf{M}(\mathtt{i})$ due to transitivity of $\mathbf{M}$.
Consequently, all entries in the matrix $[\mathrm{Dp}_{I}]$ are positive.

Since $\mathrm{Dp}_{I}^2\cong\mathrm{Dp}_{I}$, we have $[\mathrm{Dp}_{I}]=[\mathrm{Dp}_{I}]^2$.
From \cite[Proposition~6]{MM4}, we know that there exists a permutation matrix $S$ such that
the idempotent matrix $S^{-1}[\mathrm{Dp}_{I}]S$ has the following form:
\begin{displaymath}
\left(\begin{array}{ccc} 0_r & B& BC \\ 0 & 1_s & C\\ 0 & 0 & 0_t
\end{array}\right),
\end{displaymath}
where $0_r$ (resp. $0_t$) is the zero $r\times r$ (resp. $t\times t$) matrix and $1_s$ is the identity
$s\times s$ matrix such that $r+s+t=n$. Permuting the elements in $\{X_1, X_2, \dots, X_n\}$,
if necessary, we may assume that $S$ is the identity matrix. As all entries in $[\mathrm{Dp}_{I}]$ are positive,
it follows that $r=t=0$ and $s=1$, that is $[\mathrm{Dp}_{I}]=(1)$. Hence $\mathbf{M}(\mathtt{i})$
has only one indecomposable object up to isomorphism. We denote this object by $X$.

From the above we have $\mathrm{Dp}_{I}X\cong X$. Thus, for any $J\in\mathcal{I}(A)$ we
have $\mathrm{Dp}_{JI}X\cong \mathrm{Dp}_{J}X$. Therefore, for those ideals $J$ such that $JI\neq I$,
we have $[\mathrm{Dp}_{J}]=[\mathrm{Dp}_{JI}]=(0)$; and for those ideals $J$ such that $JI= I$
we have $[\mathrm{Dp}_{J}]=[\mathrm{Dp}_{JI}]=(1)$. This implies that each $\mathrm{Dp}_{J}$ induces an
endomorphism of the endomorphism algebra $\mathrm{End}(X):=B$. Since $X$ is indecomposable, the algebra
$B$ is local and its radical consists of all nilpotent elements in $B$. In particular, this radical must be
preserved by all $\mathrm{Dp}_{J}$ and hence it generates a $\mathscr{D}_A$-invariant ideal of
$\mathbf{M}(\mathtt{i})$ which does not contain any identity morphisms apart from the one for the zero object.
By the simple transitivity of $\mathbf{M}$, the radical of $B$ must be zero. This means that $B\cong \Bbbk$ and $\mathbf{M}(\mathtt{i})$ is equivalent to $\Bbbk$-mod.

Consider the unique $2$-natural transformation $\Psi:\mathbf{P}_{\mathtt{i}}(\mathtt{i})\to \mathbf{M}(\mathtt{i})$ which sends $\mathbbm{1}_{\mathtt{i}}$ to $X$. Then $\Psi$ sends $\mathrm{Dp}_{I}$ to $\mathrm{Dp}_{I}X\cong X$
and all indecomposable $1$-morphisms $F$, satisfying $F>_{\mathcal{L}}\mathrm{Dp}_{I}$, to zero since
$F\circ \mathrm{Dp}_{I}\not\cong \mathrm{Dp}_{I}$. Therefore the restriction of $\Psi$ to
$\mathbf{N}_I(\mathtt{i})$ gives a $2$-natural transformation from $\mathbf{N}_I$ to $\mathbf{M}$
which annihilates the ideal $\mathbf{I}_I$ in $\mathbf{N}_I$. Thus it induces a $2$-natural transformation
from $\mathbf{C}_{\mathcal{L}_I}$ to  $\mathbf{M}$ and the latter is an equivalence by construction.
This completes the proof.
\end{proof}

\subsection{The Drinfeld center of $\mathscr{D}_A$}
For any positive integer $s$, we denote by $1_s$ the identity $s\times s$ matrix.
Note that there is only one object $\mathtt{i}$ in the finitary $2$-category $\mathscr{D}_A$.
Using the definition of Drinfeld center given in Subsection~\ref{s2.9}, one obtains the following result:

\begin{theorem}\label{thm}
Objects of the category $\mathcal{Z}(\mathscr{D}_A)$ are finite direct sums of copies of
$(\mathbbm{1}_{\mathtt{i}},e)$, up to isomorphism. Furthermore, we have
$\mathrm{End}_{\mathcal{Z}(\mathscr{D}_A)}((\mathbbm{1}_{\mathtt{i}},e))=\Bbbk\mathrm{id}_{\mathbbm{1}_{\mathtt{i}}}$.
\end{theorem}

\begin{proof}
Let $(F,\Theta)$ be an object in $\mathcal{Z}(\mathscr{D}_A)$.  Assume that
\begin{displaymath}
F:=\bigoplus_{i=1}^n\mathrm{Dp}_{I_i},
\end{displaymath}
for some positive integer $n$ and
$I_i\in\mathcal{I}(A)^{\mathrm{ind}}$.
We would like to use Proposition~\ref{prop} to prove that $(F, \Theta)$ decomposes into a direct sum of 
certain $\{(\mathrm{Dp}_{I_i},\Phi^{(i)}\}_{i=1}^n$ in $\mathcal{Z}(\mathscr{D}_A)$,
where each natural isomorphism $\Phi^{(i)}$ is given by:
\begin{displaymath}
\Phi^{(i)}(\mathrm{Dp}_{J}):=(\mathrm{id}_{\mathrm{Dp}_{J}} \circ_0 \pi_{\mathrm{Dp}_{I_i}})
\circ_1\Theta({\mathrm{Dp}_{J}})\circ_1(\iota_{\mathrm{Dp}_{I_i}}\circ_0 \mathrm{id}_{\mathrm{Dp}_{J}}),\quad
J\in \mathcal{I}(A).
\end{displaymath} 
It suffices to show that
condition~\eqref{eq1} holds, for any $j,k\in\{1,2,\dots,n\}$ and any $1$-morphism 
$\mathrm{Dp}_{J}$, where $ J\in\mathcal{I}(A)$. This reads as follows:
for any $j\neq k$ and any $J\in \mathcal{I}(A)$, we have
\begin{equation}\label{eqn1}
(\mathrm{id}_{\mathrm{Dp}_{J}} \circ_0 \pi_{\mathrm{Dp}_{I_k}})
\circ_1\Theta({\mathrm{Dp}_{J}})\circ_1(\iota_{\mathrm{Dp}_{I_j}}\circ_0 \mathrm{id}_{\mathrm{Dp}_{J}})=0.
\end{equation}

Similarly, applying the commuting diagram~\eqref{eq01} to the $2$-morphism
$\mathrm{Dp}_{J}\hookrightarrow\mathbbm{1}_{\mathtt{i}}$ induced by $\iota_{(J,A)}$,
we have the following commuting diagram:
\begin{displaymath}
\hspace*{-1cm}
\xymatrix{\mathrm{Dp}_{I_j}\circ
\mathrm{Dp}_{J}\ar[rr]^{\iota_{\mathrm{Dp}_{I_j}}\circ_0 \mathrm{id}_{\mathrm{Dp}_{J}}\quad}\ar[d]_{\mathrm{id}_{\mathrm{Dp}_{I_j}}\circ_0\iota_{(J,A)}}
&&(\displaystyle\bigoplus_{i=1}^n\mathrm{Dp}_{I_i})\circ
\mathrm{Dp}_{J}
\ar[rr]^{\Theta(\mathrm{Dp}_{J})}\ar[d]_{\mathrm{id}_{F}\circ_0\, \iota_{(J,A)}}
&&\mathrm{Dp}_{J}\circ(\displaystyle\bigoplus_{i=1}^n \mathrm{Dp}_{I_i})
\ar[d]_{\iota_{(J,A)}\circ_0\,
\mathrm{id}_{F}}\ar[rr]^{\quad\mathrm{id}_{\mathrm{Dp}_{J}} \circ_0 \pi_{\mathrm{Dp}_{I_k}}}&&\mathrm{Dp}_{J}\circ\mathrm{Dp}_{I_k}
\ar[d]_{\iota_{(J,A)}\circ_0\,
\mathrm{id}_{\mathrm{Dp}_{I_k}}}\\
\mathrm{Dp}_{I_j}\ar[rr]^{\iota_{\mathrm{Dp}_{I_j}}}&&\displaystyle \bigoplus_{i=1}^n\mathrm{Dp}_{I_i}
\ar[rr]^{\Theta(\mathbbm{1}_{\mathtt{i}})}&&
 \displaystyle \bigoplus_{i=1}^n\mathrm{Dp}_{I_i}\ar[rr]^{\pi_{\mathrm{Dp}_{I_k}}}&&\mathrm{Dp}_{I_k}.}
\end{displaymath}
Note that $\Theta(\mathbbm{1}_{\mathtt{i}})=\mathrm{id}_F$ and 
each $\iota_{(J,A)}\circ_0\,\mathrm{id}_{\mathrm{Dp}_{I_k}}$ is injective since each 
$I_k$ is left projective because $A$ is hereditary.
As $\pi_{\mathrm{Dp}_{I_k}}\circ_1 \iota_{\mathrm{Dp}_{I_j}}=\delta_{jk}\mathrm{id}_{\mathrm{Dp}_{I_j}}$,
using the commutativity of the above diagram, we obtain that equation~\eqref{eqn1}
follows from the injectivity of each $\iota_{(J,A)}\circ_0\,
\mathrm{id}_{\mathrm{Dp}_{I_k}}$. Therefore, by Proposition~\ref{prop}, now we only need to determine 
objects $(F,\Theta)$ in $\mathcal{Z}(\mathscr{D}_A)$, for all indecomposable $1$-morphisms~$F$. This
reduces the problem to the case  $n=1$.

Assume that $(\mathrm{Dp}_{I},\Theta)$, where $I\in\mathcal{I}(A)^{\text{ind}}$, is an object 
in $\mathcal{Z}(\mathscr{D}_A)$. Then $\Theta$ is a natural isomorphism from the functor of 
$\mathrm{Dp}_{I}\circ {}_-$ to the functor ${}_-\circ \mathrm{Dp}_{I}$ given by a family of isomorphisms
\begin{displaymath}
\Theta(\mathrm{Dp}_{J}): \mathrm{Dp}_{I}\circ \mathrm{Dp}_{J}\to \mathrm{Dp}_{J}\circ \mathrm{Dp}_{I}, \quad J\in \mathcal{I}(A).
\end{displaymath}
From \cite[Corollary~8]{Zh}, for any ideals $J',J''\in\mathcal{I}(A)^{\mathrm{ind}}$, we have
\begin{equation}\label{epp1}
\mathrm{Hom}_{A\text{-}A}(J',J'')=\left\{\begin{array}{ll}\Bbbk\iota_{(J',J'')}, & \text{if } J'\subset J'';\\
0, & \text{if } J'\not\subset J'',\end{array}\right.
\end{equation}
where $\iota_{(J',J'')}$ denotes the natural inclusion.
Note that $\mathrm{Dp}_{J'}\circ \mathrm{Dp}_{J''}\cong
\mathrm{Dp}_{J'J''}$, for any two ideals $J',J''$.
Therefore we have $JI=IJ$, for any $J\in \mathcal{I}(A)$.
We claim that this implies $I=A$.
Assume that
\begin{displaymath}
G(I)=\{u_1^1, u_2^1,\dots,u_{i_1}^1\}.
\end{displaymath}

Let us first prove that all
generators in $G(I)$ are of length zero. Indeed, if this would not be the case, there would exist some
$j\in \{1,2,\ldots,i_1\}$ such that $\mathfrak{l}(u_j^1)\geq 1$. Let $J$ be the ideal generated by the element
$\varepsilon_{\mathfrak{s}(u_j^1)}$.
Note that the ideal $IJ$ can be generated by the set
\begin{displaymath}
 \{u_t^1a\varepsilon_{\mathfrak{s}(u_j^1)}|\,a\in Q^p, 1\leq t\leq i_1\},
 \end{displaymath}
 containing the set $G(IJ)$, and there is at most one path between
 any two vertices in $Q$, see~\cite[Lemma~1]{Zh}.
Since $u_j^1=u_j^1\varepsilon_{\mathfrak{s}(u_j^1)}\in IJ$ and
$u_1^1,u_2^1,\ldots,u_{i_1}^1$ are not comparable pairwise, we see that $u_j^1$ is a minimal element
with respect to $\preceq$ and thus lies in $G(IJ)$.

At the same time, it is clear that the ideal $JI$ can be generated by the set
\begin{displaymath}
\{\varepsilon_{\mathfrak{s}(u_j^1)}bu_t^1|\,b\in Q^p, 1\leq t\leq i_1\},
\end{displaymath}
which contains the set $G(JI)$.
Since $\mathfrak{l}(u_j^1)\geq 1$, we have
$\mathfrak{s}(u_j^1)\neq \mathfrak{t}(u_j^1)$ which implies $u_j^1\not\in G(JI)=G(IJ)$,
a contradiction. Therefore we get $\mathfrak{l}(u_t^1)=0$ for all $1\leq t\leq i_1$
and thus $G(I)\subset \{\varepsilon_i|\,i\in Q_0\}$.

Now we prove that $G(I)=\{\varepsilon_i|\,i\in Q_0\}$, that is, $I=A$. Otherwise, there
would exist some $j\in Q_0$ such that $\varepsilon_{j}\not\in G(I)$.
Due to indecomposability of $A$ as an $A$-$A$-bimodule, we may
assume that there exists some $j'\in \mathfrak{s}_{G(I)}$ such that there is an arrow $c$ either from
$j'$ to $j$ or from $j$ to $j'$. We consider the case where $c$ goes from $j'$ to $j$, the
other case is dealt with by similar arguments. Set $J$ to be the
ideal generated by the element $\varepsilon_j$. Then the ideal $JI$
can be generated by the set
\begin{displaymath}
\{\varepsilon_ja\,\varepsilon_i|\,i\in \mathfrak{s}_{G(I)}, a\in Q^p\}\supset G(JI)
\end{displaymath}
and the ideal $IJ$ can be generated by the set
\begin{displaymath}
\{\varepsilon_ib\,\varepsilon_j|\,i\in \mathfrak{s}_{G(I)}, b\in Q^p\}\supset G(IJ).
\end{displaymath}
Since $\mathfrak{t}(c)=j\not\in \mathfrak{s}_{G(I_1)}$, we obtain $c\not\in G(IJ)=G(JI)$ which contradicts the fact that $c=\varepsilon_{j}c\varepsilon_{j'}
\in G(JI)$.
Thus we have $G(I)=\{\varepsilon_i|\,i\in Q_0\}$ and indecomposable pairs
in $\mathcal{Z}(\mathscr{D}_A)$, up to isomorphism, are of the form
$(\mathbbm{1}_{\mathtt{i}},\Theta)$ for some natural isomorphism $\Theta$.

Due to \eqref{epp1}, for each~$J\in \mathcal{I}(A)$, we may assume that $\Theta(\mathrm{Dp}_{J})=k_J
\mathrm{id}_{\mathrm{Dp}_{J}}$, where~$k_J\in\Bbbk$.
Applying the naturality of $\Theta$ to the $2$-morphism
$\mathrm{Dp}_{J}\hookrightarrow\mathbbm{1}_{\mathtt{i}}$ induced by $\iota_{(J,A)}$,
we obtain $\iota_{(J,A)}\circ_1\Theta(\mathrm{Dp}_{J})=\iota_{(J,A)}$
which yields $k_J=1$ for all $J$.
Hence we have $\Theta=e$.

By definition, we have $\mathrm{End}_{\mathcal{Z}(\mathscr{D}_A)}((\mathbbm{1}_{\mathtt{i}},e))\subset
\mathrm{End}_{\mathscr{D}_A(\mathtt{i},\mathtt{i})}(\mathbbm{1}_{\mathtt{i}})=\Bbbk
\mathrm{id}_{\mathbbm{1}_{\mathtt{i}}}$. It is easy to check that any scalar multiple of
$\mathrm{id}_{\mathbbm{1}_{\mathtt{i}}}$ satisfies formula~\eqref{2c}. The statement follows.
\end{proof}

\section{Finitary $2$-category associated to complementary ideals for a tree algebra}\label{s4}

In this section, if not explicitly stated otherwise,
we let $A$ be the path algebra of a tree quiver $Q$ as described in Subsection~\ref{s3.1}.

\subsection{Complementary ideals for $A$}

An ideal $I$ in $A$ is said to be {\em complementary} if the projection $A\tto A/I$ splits.
Denote by $\mathcal{CI}(A)$ the set of all complementary ideals in $A$. It is clear that
$A\in\mathcal{CI}(A)$. For an ideal $I\in\mathcal{I}(A)$, we assume that $G(I)=\{u_1,u_2,\ldots,u_k\}$ and denote
the canonical map $A\tto A/I$ by $\overline{\cdot}$.

\begin{lemma}\label{l2}
For an ideal $I$ in $A$ as above, we have $I\in\mathcal{CI}(A)$ if and only if
$\mathfrak{l}(u_i)\leq 1$ for all $1\leq i\leq k$.
\end{lemma}

\begin{proof}
To prove the  ``if" part, we note that $A/I$ has a basis consisting of images of all paths in $A\backslash I$
under the canonical map $\overline{\cdot}$. Then the map $\varphi:A/I\to A$ sending $\overline{v}$ to $v$,
where $v$ runs through all paths in $A\backslash I$, splits  $\overline{\cdot}$. Indeed,
we only need to show that this map is a homomorphism. Let $v$  and $w$ in
$A\backslash I$ be such that $vw\neq 0$. We claim that $vw\not \in I$. Indeed, if the latter would not be
the case, there would exist some  $u_j$ and $a,b\in Q^p$ such that $vw=au_jb$. As $\mathfrak{l}(u_j)\leq 1$,
then either $v$ or $w$  is  comparable with $u_j$, which implies that either $v$ or $w$ lies in $I$,
a contradiction. The claim follows.

To prove the ``only if" part, let $\varphi:A/I\to A$ be a splitting of $\overline{\cdot}$.
Assume that there is some $u_j$ such that $\mathfrak{l}(u_j)>1$.
Thus $u_j$ can be written as a composition of two paths $v$ and $w$, both of nonzero length,
that is, $u_j=vw$. Due to minimality of $G(I)$, the images of $v$ and $w$ under the canonical
map $\overline{\cdot}$ are nonzero. By injectivity of $\varphi$, we have
\begin{displaymath}
0\neq \varphi(\overline{v})=\varphi(\overline{v})\varphi(\overline{\varepsilon_{\mathfrak{s}(v)}})
=\varphi(\overline{v})\varphi(\overline{\varepsilon_{\mathfrak{t}(w)}})\quad\text{and}\quad 0\neq\varphi(\overline{w})=\varphi(\overline{\varepsilon_{\mathfrak{t}(w)}})\varphi(\overline{w}).
\end{displaymath}
This implies $\varphi(\overline{vw})=\varphi(\overline{v}\cdot\overline{w})=\varphi(\overline{v})\varphi(\overline{w})\neq 0$
since $A$ is hereditary. However, we have $\varphi(\overline{vw})=0$ since $\overline{v}\cdot\overline{w}=\overline{vw}=\overline{u_j}=0$,
a contradiction. The claim follows.
\end{proof}

\begin{corollary}\label{cr8}
If $I,J\in \mathcal{CI}(A)$, then we have $I+J\in\mathcal{CI}(A)$.
\end{corollary}

\begin{proof}
This follows from Lemma~\ref{l2} and the observation  that $G(I+J)\subset G(I)\cup G(J)$.
\end{proof}

\subsection{Identity bimodules twisted by a family of endomorphisms}

For any ideal $I\in\mathcal{CI}(A)$, we denote by $\varphi_I:A\to A$ the composition of the canonical map $\overline{\cdot}: A\to A/I$ with the splitting $A/I\to A$ constructed in the proof of
Lemma~\ref{l2}. Then $\varphi_I$ is the identity on all paths in $A\backslash I$ and zero
on all paths in $I$.

Consider the identity $A$-$A$-bimodule $A={}_AA_A$. Given a unital algebra endomorphism $\varphi$ of $A$,
define a new bimodule ${}^{\varphi}A$ to be equal to $A$ as a vector space but with the bimodule action given by
\begin{displaymath}
b\cdot a\cdot c:=\varphi(b)ac \qquad\text{ for all } \quad a,b,c\in A.
\end{displaymath}
In particular, for any ideal $I\in\mathcal{CI}(A)$ such that $G(I)=\{u_1,u_2,\dots,u_k\}$ with
$\mathfrak{l}(u_i)=1$ for all $1\leq i\leq k$,  we have the corresponding $A$-$A$-bimodule
${}^{\varphi_I}A$.

Let $\mathcal{CI}^{(1)}(A)$ denote the subset of $\mathcal{CI}(A)$ consisting of all complementary ideals
generated by paths of length 1. Then the set  $\mathcal{CI}^{(1)}(A)$ has $2^{|Q_1|}$ elements.
Similarly to Corollary~\ref{cr8}, we have $I+J\in\mathcal{CI}^{(1)}(A)$ for all $I,J\in \mathcal{CI}^{(1)}(A)$.

\begin{example}
{\rm Let $A=\Bbbk Q$, where $Q$ is given by the following picture:
\begin{displaymath}
\xymatrix{1\ar[r]^\alpha&2\ar[d]_\gamma\ar[r]^\beta&3\ar[r]^\delta&5\\
&4&&}
\end{displaymath}
Set $I=\langle\beta\rangle$. The $A$-$A$-bimodule ${}^{\varphi_I}A$ decomposes into two
$A$-$A$-subbimodules as follows:
\begin{displaymath}
\xymatrix{\varepsilon_1\ar[d]&&&&&&\\
\alpha\ar[d]&\varepsilon_2\ar@{-->}[l]\ar[d]&&&\beta\alpha\ar[d]&\beta\ar@{-->}[l]\ar[d]&\varepsilon_3\ar@{-->}[l]\ar[d]&\\
\gamma\alpha&\gamma\ar@{-->}[l]&\varepsilon_4\ar@{-->}[l]&&\delta\beta\alpha&\delta\beta\ar@{-->}[l]&\delta\ar@{-->}[l]
&\varepsilon_5\ar@{-->}[l]}
\end{displaymath}
Both subbimodules are described by a basis consisting of paths. Solid arrows depict the left action while
dashed arrows depict the right action. The left subbimodule corresponds to the identity bimodule of the
subquiver $\xymatrix{1 \ar[r]^\alpha & 2\ar[r]^\gamma&4}$ and the right one
is exactly the ideal of $A$ generated by $\varepsilon_3,\varepsilon_5$.
}
\end{example}

For a tree algebra $A$ and $I\in \mathcal{CI}^{(1)}(A)$ one can calculate all indecomposable
components of the $A$-$A$-bimodule ${}^{\varphi_I}A$. However, we did not find any uniform way to describe
them in the general case. At the same time, in Subsection~\ref{s4.4} we propose such a description in the
case of the uniformly oriented Dynkin quiver of type $A_n$, where $n$ is a positive integer.

\subsection{New finitary $2$-categories for tree algebras}

Let $\mathcal{C}$ be a small category equivalent to $A$-mod. Define the $2$-category $\mathscr{D}_{\mathcal{CI}^{(1)}(A)}$ to have
\begin{itemize}
\item one object $\mathtt{i}$ (which we identify with $\mathcal{C}$);
\item as $1$-morphisms, all functors given, up to equivalence with $A$-mod, by
functors from the additive closure of the identity functor and all ${}^{\varphi_I}A\otimes_A{}_-$, where $I\in\mathcal{CI}^{(1)}(A)$ ;
\item as $2$-morphisms, all  natural transformations of functors.
\end{itemize}
To justify that this indeed defines a $2$-category, we need the following statement.

\begin{lemma}\label{ll9}
For every two ideals $I,J\in \mathcal{CI}^{(1)}(A)$, the $A$-$A$-bimodules
${}^{\varphi_I}A\otimes_A{}^{\varphi_J}A$ and ${}^{\varphi_{I+J}}A$ are isomorphic.
\end{lemma}

\begin{proof}
Note that the map $\varphi_J\circ\varphi_I$ is the identity when restricted to paths in $A\setminus(I+J)$ and zero
when restricted to paths in $I+J$. The map $\varphi_{I+J}$ has the same properties. Since $A$ has a basis consisting
of paths, we get $\varphi_J\circ\varphi_I=\varphi_{I+J}$. Let $\kappa_{(I,\,J)}$ be the map from
${}^{\varphi_I}A\otimes_A{}^{\varphi_J}A$ to ${}^{\varphi_{I+J}}A$ sending $b\otimes c$ to $\varphi_J(b)c$
for any $b,c\in A$. It is easy to check that this map is well-defined. Moreover, for any $a,a'\in A$, we have
\begin{eqnarray*}
\kappa_{(I,\,J)}(a\cdot (b\otimes c)\cdot a')&=&\kappa_{(I,\,J)}(\varphi_I(a)b\otimes ca')=\varphi_J(\varphi_I(a)b)ca'\\
&=&\varphi_J\circ\varphi_I(a)\varphi_J(b)ca'=\varphi_{I+J}(a)\varphi_J(b)ca'\\
&=&a\cdot\kappa_{(I,\,J)}(b\otimes c)\cdot a'.
\end{eqnarray*}
Therefore $\kappa_{(I,J)}$ is a $A$-$A$-bimodule homomorphism. It is straightforward to verify that
it is bijective. The claim follows.
\end{proof}

\begin{proposition}
The category $\mathscr{D}_{\mathcal{CI}^{(1)}(A)}$ is a finitary $2$-category.
\end{proposition}

\begin{proof}
By definition, $\mathscr{D}_{\mathcal{CI}^{(1)}(A)}$ has one object. Since the tree algebra $A$ is connected,
the identity $A$-$A$-bimodule $A={}^{\varphi_0}A$ is indecomposable, which means that the identity functor $\mathbbm{1}_{\mathtt{i}}\cong A\otimes_A{}_-$ is indecomposable as well. From Lemma~\ref{ll9} it follows
that $\mathscr{D}_{\mathcal{CI}^{(1)}(A)}$ is closed with respect to composition of $1$-morphisms.

Note that there are finitely many ideals $I$ in $\mathcal{CI}^{(1)}(A)$ and the tree algebra $A$ is
finite dimensional since $Q$ is finite. Each functor ${}^{\varphi_I}A\otimes_A{}_-$ thus has finitely many
direct summands. Therefore the category $\mathscr{D}_{\mathcal{CI}^{(1)}(A)}(\mathtt{i},\mathtt{i})$ has
finitely many indecomposable $1$-morphisms up to isomorphism. Since $2$-morphisms are just $A$-$A$-bimodule homomorphism between the corresponding finite dimensional $A$-$A$-bimodules, dimensions of
the corresponding spaces are all finite.
\end{proof}

\subsection{Cells in $\mathscr{D}_{\mathcal{CI}^{(1)}(A)}$ associated to a quiver of type $A_n$}\label{s4.4}

Let $Q$ be the following quiver:
\begin{equation}\label{qa}
\xymatrix{1\ar[r]^{\alpha_1}&2\ar[r]^{\alpha_2}&3\ar[r]^{\alpha_3}&\cdots\ar[r]^{\alpha_{n-1}}&n}.
\end{equation}
We have
\begin{equation}\label{ee111}
\dim \varepsilon_j A\varepsilon_i=\left\{\begin{array}{ll}1, & \text{if } i\leq j;\\
0, & \text{otherwise}.\end{array}\right.
\end{equation}
Then the identity bimodule $_AA_A$ can be depicted as the following planar graph:
\begin{equation}\label{i1}
\xymatrix{\bullet\ar[d]&&&&\\
\bullet\ar[d]&\bullet\ar@{-->}[l]\ar[d]&&&&\\
\bullet\ar[d]&\bullet\ar[d]\ar@{-->}[l]&\bullet\ar@{-->}[l]\ar[d]&&&\\
\vdots\ar[d]&\vdots\ar[d]\ar@{-->}[l]&\vdots\ar[d]\ar@{-->}[l]&\ddots\ar@{-->}[l]\ar[d]&&\\
\bullet\ar[d]&\bullet\ar[d]\ar@{-->}[l]&\cdots\ar[d]\ar@{-->}[l]&\bullet\ar@{-->}[l]\ar[d]&\bullet\ar[d]\ar@{-->}[l]&\\
\bullet&\bullet\ar@{-->}[l]&\cdots\ar@{-->}[l]&\bullet\ar@{-->}[l]&\bullet\ar@{-->}[l]
&\bullet\ar@{-->}[l]\\
}
\end{equation}
Here both the last row and the first column have n bullets and the bullet in the position $(i,j)$ stands
for the unique path from $j$ to $i$, for all $1\leq j\leq i\leq n$. The left action is depicted by solid arrows
and the right action is depicted by dashed arrows. For example, in the case $n=3$ diagram \eqref{i1}
reads as follows:
\begin{displaymath}
\xymatrix{\varepsilon_1\ar[d]_{\alpha_1\cdot}&&&&\\
\alpha_1\ar[d]_{\alpha_2\cdot}&\varepsilon_2\ar@{-->}[l]_{\cdot\alpha_1}\ar[d]_{\alpha_2\cdot}&&&&\\
\alpha_2\alpha_1&\alpha_2\ar@{-->}[l]^{\cdot\alpha_1}&\varepsilon_2\ar@{-->}[l]^{\cdot\alpha_2}&&&
}
\end{displaymath}

For each $1\leq i\leq n$, denote by $J_i$ the ideal generated by elements
$\varepsilon_i,\varepsilon_{i+1},\ldots,\varepsilon_n$. Then each $J_i$ is an indecomposable idempotent ideal
and $J_1$ is the identity $A$-$A$-bimodule $_AA_A$. Moreover, we have
\begin{displaymath}
J_1\supset J_2\supset\cdots\supset J_n.
\end{displaymath}
For any $1\leq i\leq j\leq n$, we define $M_{i,j}:=J_i/J_{j+1}$ (here $J_{n+1}:=0$),
which inherits from $J_i$ the structure  of an $A$-$A$-bimodule. It follows from the
definition that each $M_{i,j}$ is indecomposable and that $M_{i,n}=J_i$.
Note that the bimodule $M_{i,j}$ has a basis consisting of all paths listed
in the rows $i,i+1,\dots,j$ of the diagram \eqref{i1}.

For any ideal $I\in\mathcal{CI}^{(1)}(A)$, we have $G(I)\subset Q_1$.
Let us assume that
\begin{displaymath}
G(I)=\{\alpha_{i_1},\alpha_{i_2},\ldots,\alpha_{i_s}\},\quad \text{where } 1\leq i_1<i_2<\dots<i_s<n.
\end{displaymath}

\begin{lemma}\label{lemnew}
For $I$ as above, we have a decomposition
\begin{displaymath}
{}^{\varphi_{I}}A\cong M_{1,i_1}\oplus M_{i_1+1,i_2}\oplus\dots\oplus M_{i_s+1,n}.
\end{displaymath}
\end{lemma}

\begin{proof}
From the above we have that both the left hand side and the right hand side have natural
bases consisting of all paths. We claim that the identity map on the paths gives rise to
an isomorphism of bimodules. That this map is an isomorphism of right modules follows directly
by construction. That this map is an isomorphism of left modules follows from the definitions
and the observation that the left multiplication with each $\alpha_{i_m}$, where $1\leq m\leq s$,
annihilates both the left hand side and the right hand side.
\end{proof}

Informally, one can say that the decomposition of ${}^{\varphi_{I}}A$ in Lemma~\ref{lemnew}
is obtained by  cutting all $i_m$-th rows of vertical arrows in \eqref{i1}, where $1\leq m\leq s$.

For each $1\leq i\leq j\leq  n$, denote by $F_{i,j}$ the functor
\begin{displaymath}
M_{i,j}\otimes_A{}_-: A\text{-mod}\to A\text{-mod}.
\end{displaymath}
We loosely identify $F_{i,j}$ with a corresponding endofunctor of $\mathcal{C}$.
Directly from Lemmata~\ref{ll9} and \ref{lemnew} and the definitions, we obtain:

\begin{corollary}
The list
\begin{equation}\label{de1}
\{F_{i,j}\,\,:\,\,1\leq i\leq j\leq  n\}
\end{equation}
is a complete and irredundant list of indecomposable $1$-morphisms in $\mathscr{D}_{\mathcal{CI}^{(1)}(A)}$,
up to isomorphism.
\end{corollary}

Now we can explicitly describe composition of $1$-morphisms in $\mathscr{D}_{\mathcal{CI}^{(1)}(A)}$.

\begin{lemma}\label{lem13new}
For any $1\leq i\leq j\leq n$ and $1\leq i'\leq j'\leq n$, we have
\begin{displaymath}
F_{i,j}\circ F_{i',j'}\cong
\begin{cases}
F_{\max(i,i'),\min(j,j')},&  \max(i,i')\leq \min(j,j');\\
0, & \text{otherwise}.
\end{cases}
\end{displaymath}
\end{lemma}

\begin{proof}
The top of the $A$-$A$-bimodule $M_{a,b}$, where $1\leq a\leq b\leq n$, is given by
the idempotent paths $\varepsilon_{c}$ for $a\leq c\leq b$. This implies that
$M_{a,b}\otimes_A M_{a',b'}$, where $1\leq a'\leq b'\leq n$, is nonzero provided that
$[a,b]\cap [a',b']\neq0$, moreover, if some $\varepsilon_{c}$ appears in both
$M_{a,b}$ and $M_{a',b'}$, it also appears in $M_{a,b}\otimes_A M_{a',b'}$.

For $1\leq i< n$, denote by $I_i$ the ideal of $A$ generated by $\alpha_i$ and set
$I_0=I_n:=0$. Then, by Lemma~\ref{ll9}, we have
\begin{equation}\label{may24}
{}^{\varphi_{I_{a-1}}}A\otimes_A {}^{\varphi_{I_b}}A\cong {}^{\varphi_{I_{a-1}+I_b}}A.
\end{equation}
By Lemma~\ref{lemnew}, this can be written as
\begin{displaymath}
(M_{1,a-1}\oplus M_{a,n})\otimes_A (M_{1,b}\oplus M_{b+1,n})\cong
M_{1,a-1}\oplus M_{a,b}\oplus M_{b+1,n}.
\end{displaymath}
Using distributivity of the tensor product with respect to direct sums,
Krull-Schmidt property for bimodules, and taking
the previous paragraph into account, we obtain
\begin{equation}\label{may24-2}
\begin{array}{ccc}
M_{1,a-1}\otimes_A M_{1,b}&\cong& M_{1,a-1},\\
M_{a,n}\otimes_A M_{b+1,n}&\cong& M_{b+1,n},\\
M_{a,n}\otimes_A M_{1,b}&\cong& M_{a,b},\\
M_{1,a-1}\otimes_A M_{b+1,n}&=& 0.
\end{array}
\end{equation}
Swapping the factors in the left hand side of \eqref{may24} and using a similar argument,
we also obtain
\begin{equation}\label{may24-3}
\begin{array}{ccc}
M_{1,b}\otimes_A M_{1,a-1}&\cong& M_{1,a-1},\\
M_{b+1,n}\otimes_A M_{a,n}&\cong& M_{b+1,n},\\
M_{1,b}\otimes_A M_{a,n}&\cong& M_{a,b},\\
M_{b+1,n}\otimes_A M_{1,a-1}&=& 0.
\end{array}
\end{equation}

Using \eqref{may24-2} and \eqref{may24-3}, we can now compute:
\begin{displaymath}
\begin{array}{rcl}
M_{i,j}\otimes_A M_{i',j'}&\cong&
M_{1,j}\otimes_A M_{i,n}\otimes_A M_{1,j'}\otimes_A M_{i',n}\\
&\cong&
M_{1,j}\otimes_A M_{1,j'}\otimes_A M_{i,n}\otimes_A M_{i',n}\\
 &\cong&
M_{1,\min(j,j')}\otimes_A M_{\max(i,i'),n}.
\end{array}
\end{displaymath}
Now the claim follows by yet another application of \eqref{may24-2} and \eqref{may24-3}.
\end{proof}

\begin{remark}\label{rr1}
{\rm
From Lemma~\ref{lem13new}, we have:
\begin{enumerate}[$($i$)$]
\item\label{rr1a}
Each indecomposable $1$-morphism in $\mathscr{D}_{\mathcal{CI}^{(1)}(A)}$ is an idempotent.
\item\label{rr1b}
For any indecomposable $1$-morphisms $F,G$ in $\mathscr{D}_{\mathcal{CI}^{(1)}(A)}$,
we have $F\circ G\cong H\cong G\circ F$ and moreover $H$ is an indecomposable $1$-morphism.
Since all multiplicities of simple subbimodules in each $M_{i,j}$, where $1\leq i\leq j\leq n$,
are at most one, by \cite[Lemma~7]{Zh}, the endomorphism algebra of each
$M_{i,j}$ reduces to scalars. We fix a unique (up to a nonzero scalar) invertible
natural transformation $\epsilon^F$ determined by a family of isomorphisms $\epsilon^F(G):F\circ G
\to H\to G\circ F$ given by Lemma~\ref{ll9}.
 \item\label{rr1c} We have $F_{i,j}\circ F_{i',j'}\cong F_{i',j'}$ if and only if
$i\leq i'\leq j'\leq j$. We have $F_{i,j}\circ F_{i',j'}\cong F_{i,j}$ if and only if
$i'\leq i\leq j\leq j'$.
\end{enumerate}
}
\end{remark}

The following claim follows directly from the observation in
Remark~\ref{rr1}\eqref{rr1c}.

\begin{corollary}\label{cr14}
Each isomorphism class of $F_{i,j}$, where $1\leq i\leq j\leq n$, forms a two-sided cell, in particular,
also a left cell and a right cell.
\end{corollary}

\subsection{Quiver for the underlying algebra for the principal $2$-representation
of $\mathscr{D}_{\mathcal{CI}^{(1)}(A)}$ associated to a quiver of type $A_n$}

The aim of this subsection is to describe the quiver underlying the endomorphism
algebra of the additive generator for the category
$\mathscr{D}_{\mathcal{CI}^{(1)}(A)}(\mathtt{i},\mathtt{i})$, where $A$
is the path algebra of the quiver of type $A_n$  given by~\eqref{qa}.
We first have the following observation.

\begin{lemma}\label{lx1}
For any $1\leq i\leq j\leq n$ and $1\leq i'\leq j'\leq n$, we have
\begin{displaymath}
\dim \mathrm{Hom}_{A\text{-}A}(M_{i,j},M_{i',j'})\leq 1.
\end{displaymath}
Moreover, the equality holds if and only if $i'\leq i\leq j'\leq j$.
\end{lemma}

\begin{proof}
By~\eqref{ee111}, we see that all composition multiplicities of $A$, viewed
as an $A$-$A$-bimodule, are at most $1$. The same holds for each $M_{i,j}$.
If $\mathrm{Hom}_{A\text{-}A}(M_{i,j},M_{i',j'})\neq 0$, then $i\leq j'$ since,
otherwise, $M_{i,j}$ and $M_{i',j'}$ would not have any composition subquotients
in common. If $j<j'$, then $\mathrm{Hom}_{A\text{-}A}(M_{i,j},M_{i',j'})=0$
since in this case $M_{i,j}$ does not contain a composition subquotient
isomorphic to the simple socle of $M_{i',j'}$. Finally, if
$i\leq j'\leq j$, then any map from $M_{i,j}$ to $M_{i',j'}$ factors through
$M_{i,j'}$ as all remaining simple subquotients of $M_{i,j}$ are automatically
in the kernel. Thus the assertion of this lemma  follows directly from~\cite[Lemma~7]{Zh}.
\end{proof}

For any $i'\leq i\leq j'\leq j$, the obvious inclusion of $J_i$ into $J_{i'}$
induces a nonzero map from $M_{i,j}$ to $M_{i',j'}$ which we denoted by
$\varsigma_{(M_{i,j},\,M_{i',j'})}$. By Lemma~\ref{lx1} ,
the map $\varsigma_{(M_{i,j},\,M_{i',j'})}$
forms a basis of the corresponding homomorphism space.

Now we can determine the quiver $\mathcal{Q}^{\mathcal{CI}^{(1)}}$ for the
underlying algebra of the principal $2$-representations $\mathbf{P}_{\mathtt{i}}$
of $\mathscr{D}_{\mathcal{CI}^{(1)}(A)}$. The vertices of
$\mathcal{Q}^{\mathcal{CI}^{(1)}}$ are given by all indecomposable
$A$-$A$-bimodules $M_{i,j}$, where $1\leq i\leq j\leq n$. There is exactly one
arrow from each $M_{i,j}$ to $M_{i,j+1}$ corresponding to
$\varsigma_{(M_{i,j+1},\,M_{i,j})}$ and one arrow from $M_{i,j}$ to $M_{i+1,j}$
corresponding to $\varsigma_{(M_{i+1,j},\,M_{i,j})}$. The relations satisfied
by these maps are the obvious commutativity relations and zero relations,
when applicable, as indicated by the dashed arrows on the following picture:
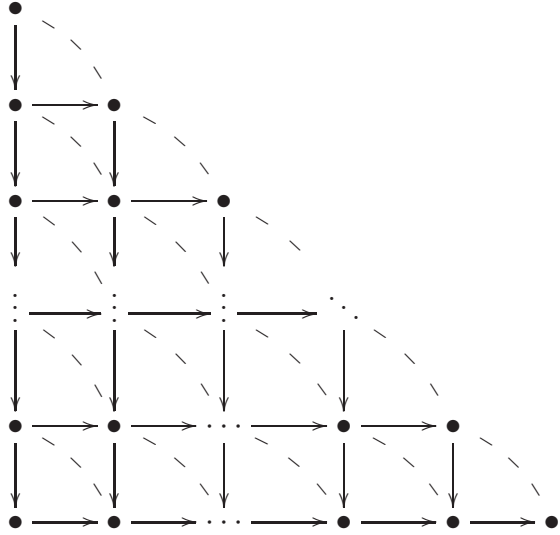
\begin{equation}\label{ii1}
\xymatrix{\bullet\ar[d]\ar@{--}@/^/[dr]&&&&\\
\bullet\ar[d]\ar[r]\ar@{--}@/^/[dr]&\bullet\ar[d]\ar@{--}@/^/[dr]&&&&\\
\bullet\ar[d]\ar[r]\ar@{--}@/^/[dr]&\bullet\ar[d]\ar[r]\ar@{--}@/^/[dr]&
\bullet\ar[d]\ar@{--}@/^/[dr]&&&\\
\vdots\ar[d]\ar[r]\ar@{--}@/^/[dr]&\vdots\ar[d]\ar[r]\ar@{--}@/^/[dr]&
\vdots\ar[d]\ar[r]\ar@{--}@/^/[dr]&\ddots\ar[d]\ar@{--}@/^/[dr]&&\\
\bullet\ar[d]\ar[r]\ar@{--}@/^/[dr]&\bullet\ar[d]\ar[r]\ar@{--}@/^/[dr]&\cdots\ar[d]\ar[r]\ar@{--}@/^/[dr]&
\bullet\ar[r]\ar[d]\ar@{--}@/^/[dr]&\bullet\ar[d]\ar@{--}@/^/[dr]&\\
\bullet\ar[r]&\bullet\ar[r]&\cdots\ar[r]&\bullet\ar[r]&\bullet\ar[r]
&\bullet
}
\end{equation}
Here both the last row and the first column have $n$ bullets and the bullet in the position
$(i,j)$ stands for the $A$-$A$-bimodule $M_{j,i}$, where $1\leq j\leq i\leq n$.
All squares commute and all top diagonal compositions are zero.

\begin{remark}
{\rm
\begin{enumerate}[$($i$)$]
\item We observe that the quiver~\eqref{ii1} is exactly the Auslander-Reiten quiver
for the original algebra $A$. Thus our construction makes the module category of
$A$ into a tensor category whose tensor structure corresponds to
that of $\mathscr{D}_{\mathcal{CI}^{(1)}(A)}(\mathtt{i},\mathtt{i})$.
 A similar tensor structure appeared from a completely
different problem considered in \cite{He}.
\item From \eqref{ii1} we have a nice  combinatorial rule for composition of
indecomposable $1$-morphisms in our $2$-category: taking two vertices in the
quiver~\eqref{ii1}, by Lemma~\ref{lem13new}, the composition of the corresponding
indecomposable $1$-morphisms (which does not depend on the order in which we compose)
is the indecomposable $1$-morphism corresponding to the
the intersection of the horizontal line going through the higher of the two vertices
and the vertical line going through the rightmost of the two vertices, if this intersection
is inside our quiver. If the intersection happens to be outside of our quiver, then the
composition is zero.
\end{enumerate}
}
\end{remark}

\subsection{Simple transitive $2$-representations for $\mathscr{D}_{\mathcal{CI}^{(1)}(A)}$}

For each indecomposable $1$-morphism $G$, denote by $\mathcal{L}_{G}$ the corresponding
left cell (consisting of the isomorphism class of $G$). By definition, we have
\begin{multline*}
\mathbf{N}_G(\mathtt{i})=\mathrm{add}(\{F:
F\text{ runs through all $1$-morphisms corresponding to vertices in the }\\
\text{upper-right area of the vertex to which }
G \text{ corresponds in the quiver~\eqref{ii1}}\})
\end{multline*}
and the ideal $\mathbf{I}_G$ in $\mathbf{N}_G$ is generated by all $2$-morphisms
$\mathrm{id}_F$, where $F\in\mathbf{N}_G(\mathtt{i})$ and $F\not\cong G$.
By Lemma~\ref{lx1}, we have $\mathrm{End}(G)\cong \Bbbk \mathrm{id}_{G}$ in the quotient category $\mathbf{N}_G/\mathbf{I}_G(\mathtt{i})=\mathbf{C}_{\mathcal{L}_G}(\mathtt{i})$ and thus we obtain $\mathbf{C}_{\mathcal{L}_G}(\mathtt{i})\cong \Bbbk$-mod.

For any indecomposable $1$-morphism $F$, define $\mathcal{ST}_F$ to be the set
consisting of all indecomposable $1$-morphisms $H$, up to isomorphism,
such that $H\circ F\cong F$.

\begin{proposition}
For any two nonisomorphic indecomposable $1$-morphisms $F$ and $G$,
the cell $2$-representations $\mathbf{C}_{\mathcal{L}_F}$ and
$\mathbf{C}_{\mathcal{L}_G}$ are not equivalent.
\end{proposition}

\begin{proof}
It follows from Remark~\ref{rr1}~\eqref{rr1c} that, under our assumptions,
there is an  indecomposable $1$-morphism $H$ such that $H\in \mathcal{ST}_{F}$
and $H\not\in \mathcal{ST}_{G}$. The claim of the proposition follows.
\end{proof}

\begin{theorem}
Every simple transitive $2$-representation of $\mathscr{D}_{\mathcal{CI}^{(1)}(A)}$ is equivalent to $\mathbf{C}_{\mathcal{L}_G}$ for some indecomposable $1$-morphism $G$.
\end{theorem}

\begin{proof}
Let $\mathbf{M}$ be a simple transitive $2$-representation of $\mathscr{D}_{\mathcal{CI}^{(1)}(A)}$. Set
\begin{displaymath}
\Sigma:=\{F\in \mathcal{S}_{\mathscr{D}_{\mathcal{CI}^{(1)}(A)}}|\,\mathbf{M}(F)\neq 0\}.
\end{displaymath}
Since $\mathbf{M}(\mathbbm{1}_{\mathtt{i}})=\mathrm{id}_{\mathbf{M}(\mathtt{i})}\neq 0$,
we see that the set $\Sigma$ is non-empty. Let $G$ be a maximal element in $\Sigma$
with respect to $\geq_{L}$. Then the additive closure of $GX$,
where $X$ runs through all objects in $\mathbf{M}(\mathtt{i})$, is non-zero
since $G\in \Sigma$, and is closed under the action of
$\mathscr{D}_{\mathcal{CI}^{(1)}(A)}$ by maximality of $G$.
Transitivity of $\mathbf{M}$ hence implies that this additive
closure must coincide with the whole of $\mathbf{M}(\mathtt{i})$.
For any indecomposable $1$-morphism $F$,
from the maximality of $G$ it follows that $F$ acts as zero on
$\mathbf{M}(\mathtt{i})$ if and only if $F\circ G\not \cong G$.
In particular, if $F\circ G\not \cong G$, then none of the
direct summands of $F\circ G$ lies in $\Sigma$.

Assume that $X_1, X_2, \dots, X_n$ is a complete and irredundant
list of representatives of isomorphism classes of indecomposable
objects in $\mathbf{M}(\mathtt{i})$. Since $G\in \Sigma$, there
exists some $j$ such that $GX_j\neq 0$. Note that $0\neq\mathrm{add}(GX_j)$
is $\mathscr{D}_{\mathcal{CI}^{(1)}(A)}$-invariant. Due to transitivity
of $\mathbf{M}$, we obtain $\mathrm{add}(GX_j)=\mathbf{M}(\mathtt{i})$.

Now we claim that the set $\Sigma$ has a unique maximal element $G$
with respect to $\geq_{L}$. Indeed, if $H$ would be another such
maximal element, then maximality of both $G$ and $H$ would imply
$H\circ G\cong G\circ H\not\cong G$ which, by the above, would mean that
$H\not\in \Sigma$, a contradiction. Therefore, for any $H\in \Sigma$,
we have $H\leq_L G$ and hence $H\circ G\cong G\cong G\circ H$.

Next we claim that $GX_i\neq 0$ for all $i$. Indeed, assume $GX_{i}=0$
for some $i$. Then, for any $F\in \Sigma$, we have $0=GX_{i}=G\circ FX_{i}$
(for the second equality we use $G\circ F=G$ for $F\in \Sigma$ which was
established in the previous paragraph).
This means that $\mathbf{G}_\mathbf{M}(\{X_{i}\})$ is annihilated by $G$
and hence cannot coincide with $\mathbf{M}(\mathtt{i})$ since $G\in \Sigma$.
This, however, contradicts transitivity of $\mathbf{M}$. Therefore $GX_i\neq 0$ for all $i$,
moreover, $\mathrm{add}(GX_i)$ is $\mathscr{D}_{\mathcal{CI}^{(1)}(A)}$-invariant for each $i$,
and thus must coincide with $\mathbf{M}(\mathtt{i})$ due to transitivity of $\mathbf{M}$.
Consequently, all entries in the matrix $[G]$ are positive.

From Remark~\ref{rr1}~\eqref{rr1a}, we see that $G$ is an idempotent.
Following the proof of Theorem~\ref{th5}, we also get $[G]=(1)$ and thus
$\mathbf{M}(\mathtt{i})$ has only one indecomposable object up to isomorphism,
denoted by $X$. For any indecomposable $1$-morphism $F$, we have $F\circ GX\cong FX$.
Therefore, if $F\in\Sigma$, then $FX\cong X$  since $F\circ G\cong G$.
If $F\not \in\Sigma$, then $FX\cong 0$.
This implies that each $F$ induces an endomorphism of the algebra
$B:=\mathrm{End}(X)$. Since $X$ is indecomposable, the algebra $B$ is local and
its radical consists of all nilpotents in $B$. Note that the radical must be
preserved by all $F$ and thus it generates a
$\mathscr{D}_{\mathcal{CI}^{(1)}(A)}$-invariant ideal of $\mathbf{M}(\mathtt{i})$,
which does not contain any identity morphisms apart from the one for the zero object.
By the simple transitivity of $\mathbf{M}$, we have $\mathrm{Rad}\, B=0$.
This means that $B\cong \Bbbk$ and $\mathbf{M}(\mathtt{i})$ is equivalent to $\Bbbk$-mod.

Consider the unique $2$-natural transformation
$\Psi:\mathbf{P}_{\mathtt{i}}(\mathtt{i})\to \mathbf{M}(\mathtt{i})$
which sends $\mathbbm{1}_{\mathtt{i}}$ to $X$. Then $\Psi$ sends $G$ to $GX\cong X$
and all indecomposable $1$-morphisms $F$ satisfying $F>_{\mathcal{L}}G$ to zero since
$F\circ G\not\cong G$. Therefore the restriction of $\Psi$ to
$\mathbf{N}_G(\mathtt{i})$ gives a $2$-natural transformation from $\mathbf{N}_G$ to $\mathbf{M}$
which annihilates the ideal $\mathbf{I}_G$ in $\mathbf{N}_G$. Thus it induces a $2$-natural transformation
from $\mathbf{C}_{\mathcal{L}_G}$ to  $\mathbf{M}$ and the latter is an equivalence by construction.
This completes the proof.
\end{proof}

\subsection{The Drinfeld center of $\mathscr{D}_{\mathcal{CI}^{(1)}(A)}$}

It is easily checked by definition that for any indecomposable $1$-morphism
$F$ the pair $(F,\epsilon_F)$ is an object in the Drinfeld center of the $2$-category $\mathscr{D}_{\mathcal{CI}^{(1)}(A)}$. Before stating the main result of this subsection,
we start with some preparations. By Lemma~\ref{lx1}, we see that
\begin{displaymath}
\dim \mathrm{Hom}_{A\text{-}A}(M_{i,j},M_{i',j'})=1\Longleftrightarrow 1\leq i'\leq
i\leq j'\leq j\leq n.
\end{displaymath}
Assume that this homomorphism space is nonzero and
\begin{displaymath}
\dim M_{i,j'}=r\geq 0,\quad \dim M_{i,j}=s\geq r\quad\text{and}\quad \dim M_{i',j'}=t\geq r.
\end{displaymath}
Each $M_{s,t}$ has a natural basis consisting of paths. Let $\varsigma_{(M_{i,j},\,M_{i',j'})}$
be the non-zero element in $\mathrm{Hom}_{A\text{-}A}(M_{i,j},M_{i',j'})$
which has the following matrix with respect to these bases:
\begin{displaymath}
\left(
\begin{array}{c|c}
1_r &0 \\ \hline
0 &0
\end{array}\right)_{s\times t}.
\end{displaymath}

Now we give a full description of the Drinfeld center $\mathcal{Z}(\mathscr{D}_{\mathcal{CI}^{(1)}(A)})$.

\begin{theorem}\label{thm1}
Objects of the category $\mathcal{Z}(\mathscr{D}_{\mathcal{CI}^{(1)}(A)})$
are finite direct sums of copies of $(F,\epsilon^F)$, up to isomorphism,
where $F$ runs through all indecomposable $1$-morphisms
and each $\epsilon^F$ is defined as in Remark~\ref{rr1}~\eqref{rr1b}.

Moreover, we have
$\mathrm{End}_{\mathcal{Z}(\mathscr{D}_{\mathcal{CI}^{(1)}(A)})}((F,\epsilon^F))
=\Bbbk\mathrm{id}_{F}$ and
\begin{equation}\label{m1}
\mathrm{Hom}_{\mathcal{Z}(\mathscr{D}_{\mathcal{CI}^{(1)}(A)})}((F,\epsilon^F),(F',\epsilon^{F'}))
=\mathrm{Hom}_{\mathscr{D}_{\mathcal{CI}^{(1)}(A)}}(F,F'),
\end{equation}
for any two pairs $(F,\epsilon^F)$ and  $(F',\epsilon^{F'})$,
where $F$ and $F'$ are not isomorphic to each other.
The category $\mathcal{Z}(\mathscr{D}_{\mathcal{CI}^{(1)}(A)})$ is
biequivalent to the category $\mathscr{D}_{\mathcal{CI}^{(1)}(A)}(\mathtt{i},\mathtt{i})$.
\end{theorem}

\begin{proof}
Let $(F,\Theta)$ be an object in $\mathcal{Z}(\mathscr{D}_{\mathcal{CI}^{(1)}(A)})$.  
Assume that
\begin{displaymath}
F:=\bigoplus_{i=1}^mF_i,
\end{displaymath}
where each $F_i$ is an indecomposable $1$-morphism.
We would like to use Proposition~\ref{prop} to prove that $(F, \Theta)$ decomposes into a direct sum of
certain $\{(F_i,\Phi^{(i)}\}_{i=1}^n$ in 
$\mathcal{Z}(\mathscr{D}_{\mathcal{CI}^{(1)}(A)})$,
where each natural isomorphism $\Phi^{(i)}$ is given by:
\begin{displaymath}
\Phi^{(i)}(K):=(\mathrm{id}_K \circ_0 \pi_{F_i})
\circ_1\Theta(K)\circ_1(\iota_{F_i}\circ_0 \mathrm{id}_{K}), \quad K\in\mathscr{D}_{\mathcal{CI}^{(1)}(A)}(\mathtt{i},\mathtt{i}).
\end{displaymath}
It suffices to show that
condition~\eqref{eq2} holds, for any $j,k\in\{1,2,\dots,m\}$ and any $1$-morphism
$K\in\mathscr{D}_{\mathcal{CI}^{(1)}(A)}(\mathtt{i},\mathtt{i})$. 
This reads as follows:
for any $j\neq k$ and any $1$-morphism $K$, we have
\begin{equation}\label{eqn2}
(\pi_{F_k}\circ_0 \mathrm{id}_{K})
\circ_1(\Theta(K))^{-1}\circ_1(\mathrm{id}_K\circ_0\iota_{F_j})=0.
\end{equation}

Similarly, applying the commutative diagram~\eqref{eq3} to each $2$-morphism
$F_{p,n}\hookrightarrow \mathbbm{1}_{\mathtt{i}}$, where $1\leq p\leq n$,
induced by $\iota_{(J_p, A)}$,
we have the following commutative diagram:
\begin{displaymath}
\hspace{-.3cm}
\xymatrix@R=3.3pc@C=2.3pc{F_{p,n} \circ F_j\ar[rr]^{\mathrm{id}_{F_{p,n}}\circ_0 \iota_{F_j}\quad}
\ar[d]_{\iota_{(J_p, A)}\circ_0\mathrm{id}_{F_j}}
&&F_{p,n} \circ(\displaystyle\bigoplus_{i=1}^mF_i)
\ar[rr]^{(\Theta(F_{p,n}))^{-1}}\ar[d]_{\iota_{(J_p, A)}\circ_0\mathrm{id}_{F}}
&& (\displaystyle\bigoplus_{i=1}^mF_i)\circ F_{p,n}
\ar[d]_{\mathrm{id}_{F}\circ_0\iota_{(J_p, A)}}\ar[rr]^{\quad \pi_{F_k}\circ_0\mathrm{id}_{K}}
&&  F_k\circ F_{p,n}
\ar[d]_{\mathrm{id}_{F_k}\circ_0\iota_{(J_p, A)}}\\
F_j\ar[rr]^{\iota_{F_j}\quad}&&\displaystyle\bigoplus_{i=1}^mF_i
\ar[rr]^{(\Theta(\mathbbm{1}_{\mathtt{i}}))^{-1}}&&
\displaystyle\bigoplus_{i=1}^mF_i\ar[rr]^{\quad\pi_{F_k}}&& F_k.}
\end{displaymath}
Note that $(\Theta(\mathbbm{1}_{\mathtt{i}}))^{-1}=\mathrm{id}_F$.
Each $M_{p,q}$, where $1\leq p\leq q \leq n$ is right projective and hence 
the functor $\mathrm{id}_{F_{p,q}}\circ_0{}_-$ is exact. Therefore
$\mathrm{id}_{F_k}\circ_0\iota_{(J_p, A)}$ is injective.
Since $\pi_{F_k}\circ_1 \iota_{F_i}=\delta_{jk}\mathrm{id}_{F_j}$,
using the commutativity of the above diagram, we obtain that equation~\eqref{eqn2}, for each $F_{p,n}$,
follows from the injectivity of each $\mathrm{id}_{F_k}\circ_0\iota_{(J_p, A)}$.

Then, applying the commutative diagram~\eqref{eq3} to 
each $2$-morphism $F_{p,n}\tto F_{p,q}$, where $ 1\leq p\leq q\leq n$,
induced by the canonical map $\varsigma_{(M_{p,n}, M_{p,q})}: J_p\tto J_p/J_{q+1}$,
we have the following commutative diagram:
\begin{displaymath}
\hspace*{-1cm}
\xymatrix@R=3.3pc@C=2.3pc{ F_{p,n} \circ F_j\ar[rr]^{\mathrm{id}_{F_{p,n}}\circ_0 \iota_{F_j}\quad}
\ar[d]_{\varsigma_{(M_{p,n}, M_{p,q})}\circ_0\mathrm{id}_{F_j}}
&&F_{p,n} \circ(\displaystyle\bigoplus_{i=1}^mF_i)
\ar[rr]^{(\Theta(F_{p,n}))^{-1}}\ar[d]_{\varsigma_{(M_{p,n}, M_{p,q})}\circ_0\mathrm{id}_{F}}
&& (\displaystyle\bigoplus_{i=1}^mF_i)\circ F_{p,n}
\ar[d]_{\mathrm{id}_{F}\circ_0\varsigma_{(M_{p,n}, M_{p,q})}}\ar[rr]^{\quad \pi_{F_k}\circ_0\mathrm{id}_{F_{p,n}}}
&&  F_k\circ F_{p,n}
\ar[d]_{\mathrm{id}_{F_k}\circ_0\varsigma_{(M_{p,n}, M_{p,q})}}\\
F_{p,q} \circ F_j\ar[rr]^{\mathrm{id}_{F_{p,q}}\circ_0 \iota_{F_j}\quad}&&F_{p,q} \circ(\displaystyle\bigoplus_{i=1}^mF_i)
\ar[rr]^{(\Theta(F_{p,q}))^{-1}}&&
(\displaystyle\bigoplus_{i=1}^mF_i)\circ F_{p,q}\ar[rr]^{\quad\pi_{F_k}\circ_0\mathrm{id}_{F_{p,q}}}&& F_k\circ F_{p,q}.}
\end{displaymath}
Note that each $\varsigma_{(M_{p,n}, M_{p,q})}\circ_0\mathrm{id}_{F_j}$ is surjective as 
tensor functors are right exact.
From the previous paragraph and the commutativity of the above diagram, we 
obtain that equation~\eqref{eqn2}, for each $F_{p,q}$,
follows from the surjectivity of each $\varsigma_{(M_{p,n}, M_{p,q})}\circ_0\mathrm{id}_{F_j}$.
Therefore, by Proposition~\ref{prop}, now we only need to determine 
objects $(F,\Theta)$ in $\mathcal{Z}(\mathscr{D}_{\mathcal{CI}^{(1)}(A)})$, for all indecomposable $1$-morphisms~$F$. This
reduces the problem to the case $m=1$.

Assume that $F$ is indecomposable,  
then $\Theta$ is a natural isomorphism from the functor of
$F\circ {}_-$ to the functor ${}_-\circ F$ given by a family of isomorphisms
\begin{displaymath}
\Theta(K): F\circ K\to K\circ F, \quad K\in \mathscr{D}_{\mathcal{CI}^{(1)}(A)}(\mathtt{i},\mathtt{i}).
\end{displaymath}
By Remark~\ref{rr1}~\eqref{rr1b} and Lemma~\ref{lx1}, for any indecomposable $1$-morphism $K$
there exist some $p,q\in \{1,2,\dots, n\}$ such that $F\circ K\cong F_{p,q}\cong K\circ F$.

Assume that the $A$-$A$-bimodule identified with each $F\circ F_{p,n}\cong F_{p,n}\circ
F$ has dimension $s_p\in\mathbbm{Z}_{>0}$ and the $A$-$A$-bimodule identified with
$F$ has dimension $t\geq s_p$. We note that, as usual, the action of morphism on modules is the right action
and $\Theta(\mathbbm{1}_{\mathtt{i}})=\mathrm{id}_F$.
Applying the naturality of $\Theta$ to the $2$-morphism
each $2$-morphism $F_{p,n}\hookrightarrow \mathbbm{1}_{\mathtt{i}}$, where~$1\leq p\leq n$,
induced by $\iota_{(J_p, A)}$,
in an appropriate basis, we have the following matrix identity:
\begin{displaymath}
\mathtt{C}_{p,n}\cdot (1_{s_p}|0)_{s_p\times t}=(1_{s_p}|0)_{s_p\times t}
\end{displaymath}
where $\mathtt{C}_{p,n}$ is, of size $s_p\times s_p$, the matrix for the isomorphism $\Theta(F_{p,n})$.
Therefore we get $\mathtt{C}_{p,n}=1_{s_p}$ and $\Theta(F_{p,n})=\epsilon^{F}(F_{p,n})$ for all $1\leq p\leq n$.

Assume that the $A$-$A$-bimodule identified with each $F\circ F_{p,q}\cong F_{p,q}\circ
F$ has dimension $s_{pq}\in\mathbbm{Z}_{>0}$, we have $s_{pq}\leq s_p$.
Similarly, applying the naturality of $\Theta$ to
each $2$-morphism $F_{p,n}\tto F_{p,q}$, where~$1\leq p\leq q\leq n$,
induced by the canonical map $\varsigma_{(M_{p,n}, M_{p,q})}: J_p\tto J_p/J_{q+1}$,
in an appropriate basis, we have the following matrix identity:
\begin{displaymath}
 \mathtt{C}_{p,n}\cdot \left(
  \begin{array}{c}
  1_{s_{pq}}  \\ \hline
  0
  \end{array}\right)_{s_{p}\times s_{pq}}=\left(
  \begin{array}{c}
  1_{s_{pq}}  \\ \hline
  0
  \end{array}\right)_{s_{p}\times s_{pq}}\cdot \mathtt{C}_{p,q}
\end{displaymath}
where $\mathtt{C}_{p,q}$ is, of size $s_{pq}\times s_{pq}$, the matrix for the isomorphism $\Theta(F_{p,q})$.
Therefore we get $\mathtt{C}_{p,q}=1_{s_{pq}}$ and $\Theta(F_{p,q})=\epsilon^{F}(F_{p,q})$
for all $1\leq p\leq q\leq n$. Hence we have $\Theta=\epsilon^{F}$.

By definition, for any indecomposable $1$-morphism $F$, we have
\begin{displaymath}
\mathrm{End}_{\mathcal{Z}(\mathscr{D}_A)}((F,\epsilon^F))\subset
\mathrm{End}_{\mathscr{D}_{\mathcal{CI}^{(1)}(A)}}(F)=\Bbbk \mathrm{id}_{F}.
\end{displaymath}
It is clear that any scalar multiple of $\mathrm{id}_{F}$ satisfies the formula~\eqref{2c}.
Thus the above embedding is, in fact, an equality.
For any pair of nonisomorphic indecomposable $1$-morphisms $F$ and $F'$, we also have
\begin{displaymath}
\mathrm{Hom}_{\mathcal{Z}(\mathscr{D}_A)}((F,\epsilon^F),(F',\epsilon^{F'}))\subset \mathrm{Hom}_{\mathscr{D}_{\mathcal{CI}^{(1)}(A)}}(F,F'),
\end{displaymath}
where the right hand side has dimension at most $1$ by Lemma~\ref{lx1}. Because of commutativity of
composition of $1$-morphisms in $\mathscr{D}_{\mathcal{CI}^{(1)}(A)}(\mathtt{i},\mathtt{i})$,
it is sufficient to prove
\begin{displaymath}
\varsigma_{(M_{i,j},\,M_{i',j'})}\circ_0 \mathrm{id}_F=\mathrm{id}_F\circ_0\varsigma_{(M_{i,j},\,M_{i',j'})}
\end{displaymath}
for any indecomposable $1$-morphism $F$ and any $i'\leq i\leq j'\leq j$. This is easily checked
using the definition of $\varsigma_{(M_{i,j},\,M_{i',j'})}$. Therefore we get equality~\eqref{m1}.

The forgetful functor from the Drinfeld center $\mathcal{Z}(\mathscr{D}_{\mathcal{CI}^{(1)}(A)})$
to the original $2$-category $\mathscr{D}_{\mathcal{CI}^{(1)}(A)}(\mathtt{i},\mathtt{i})$, which sends $(F,\epsilon^F)$ to $F$, is fully faithful and thus a biequivalence.
\end{proof}

\section{Twisted identity bimodules for truncated polynomial algebras}\label{s5}

\subsection{The fiat $2$-category $\mathscr{D}$}

Let $k,d$ be two positive integers and $\zeta\in \mathds{C}$ be a primitive $d$-th root of unity.
Set $D=\mathds{C}[x]/(x^k)$. For each $i\in\mathbb{Z}_{\geq 0}$, denote by $\varphi_i$ the
algebra automorphism of $D$ sending $x$ to $\zeta^ix$. From the definition we have
$\varphi_i\varphi_j=\varphi_{i+j}=\varphi_j\varphi_i$ for any $i,j\in \mathbb{Z}_{\geq 0}$.
If $k=1$, then $D\cong \mathds{C}$ and its endomorphism algebra only consists of scalars
of the identity homomorphism. If $d=1$, then $\zeta=1$ and all $\varphi_i$ are equal to the
identity homomorphism. Note that $D$ is a natural $D$-$D$-bimodule via left and right multiplications.
Twisting the left multiplication by the automorphism $\varphi_i$,
we get a new $D$-$D$-bimodule structure of $D$ as follows:
\begin{displaymath}
u\cdot v \cdot w=\varphi_i(u)vw,
\end{displaymath}
where $u,v,w\in D$. Denote this new $D$-$D$-bimodule by ${}^{\varphi_i}D$.
If we have either $k=1$ or $d=1$, then ${}^{\varphi_i}D\cong D$ as $D$-$D$-bimodules.
Therefore, from now on we assume that both $k,d>1$.
Since the order of $\zeta$ is $d$, then we have $\varphi_i=\varphi_j$ if $i\equiv j\ (\mathrm{mod}\ d)$
and, moreover, ${}^{\varphi_i}D={}^{\varphi_j}D$ in this case.

For each $i\in\mathbb{Z}_{\geq 0}$, denote by $F_i$ the endofunctor of $D$-mod defined as follows:
given a $D$-module $M$, the module $F_i(M)$ is equal to $M$ as a vector space, while
the action of $D$ on $F_i(M)$ is twisted by $\varphi_i$:
\begin{displaymath}
u\cdot m:= \varphi_i(u)m,\quad\text{ where }\quad m\in M, u\in D.
\end{displaymath}
Note that $F_i$ is isomorphic to the functor ${}^{\varphi_i}D\otimes_{D}{}_-$.
We also note that the functor $F_1^d$ is equal (and not only isomorphic) to $F_0$.

Define the $2$-category $\mathscr{D}$ to have
\begin{itemize}
\item one object $\mathtt{i}$ (which we identify with $D$-mod);
\item as $1$-morphisms, all possible direct sums of the $F_i$'s;
\item as $2$-morphisms, all  natural transformations of functors.
\end{itemize}
The category $\mathscr{D}$ is, clearly, a finitary $2$-category.
In fact, it is a fiat $2$-category where the weak involution $\ast$ sends the
functor $F_i$ to its inverse (and hence also biadjoint) functor
$F_{d-i}$. This category is a non-trivial generalization of \cite[Subsection~3.2]{MM5} in the case
of a cyclic group.

\subsection{Quiver for the underlying algebra for the principal $2$-representation of $\mathscr{D}$}

For any $i\in \mathbb{Z}_{\geq 0}$, we denote by $q_i:{}^{\varphi_i}D\to {}^{\varphi_i}D$
the $D$-$D$-bimodule homomorphism sending $1$ to $x$.
For any $i\not\equiv j\ (\mathrm{mod}\ d)$, we denote by $p_{ij}:{}^{\varphi_i}D\to {}^{\varphi_j}D$
the $D$-$D$-bimodule homomorphism sending $1$ to $x^{k-1}$.
From the definition we immediately have $q_i^t=0$ for any positive integer
$t\geq k$ and, also,  $p_{ij} p_{st}=0$ whenever the composition makes sense.

\begin{lemma}\label{l7}
For any $i,j\in\mathbb{Z}_{\geq 0}$, we have
\begin{displaymath}
\mathrm{Hom}_{D\text{-}D}({}^{\varphi_i}D,{}^{\varphi_j}D)=\left\{\begin{array}{ll} D,
& if\ i\equiv j\ (\mathrm{mod}\ d);\\ \mathds{C} p_{ij}, & if\ i\not\equiv j\ (\mathrm{mod}\ d).
\end{array}\right.
\end{displaymath}
\end{lemma}

\begin{proof}
Let $f:{}^{\varphi_i}D\to {}^{\varphi_j}D$ be a non-zero $D$-$D$-bimodule
homomorphism. Note that ${}^{\varphi_i}D$ is generated by the identity element
as a $D$-$D$-bimodule.
Hence $f$ is uniquely determined by $f(1)$, which satisfies
\begin{equation}\label{i}
f(1)\zeta^i x=f(1\cdot \zeta^ix)=f(\zeta^ix)=f(x\cdot 1)=x\cdot f(1)=\zeta^jxf(1).
\end{equation}
Since $\zeta$ is a primitive $d$-th root of unity, we have $\zeta^i\neq\zeta^j$
for $i\not\equiv j\ (\mathrm{mod}\ d)$.
Then the equation \eqref{i} implies $xf(1)=0$, that is $f(1)\in \mathds{C} x^{k-1}$.
For $i\equiv j\ (\mathrm{mod}\ d)$, we have $\zeta^i=\zeta^j$ and equation~\eqref{i}
holds automatically in this case. As $f\neq 0$, we can choose all nonzero element in $D$ to be $f(1)$.
This claim follows.
\end{proof}

By this lemma, we know that the set $\{\mathrm{id}_{{}^{\varphi_i}D},q_i,q_i^2,\ldots,q_i^{k-1}\}$
forms a basis of the endomorphism algebra $\mathrm{End}_{D\text{-}D}({}^{\varphi_i}D)$. Therefore $\mathrm{End}_{D\text{-}D}({}^{\varphi_i}D)$ is generated by the identity
element $\mathrm{id}_{{}^{\varphi_i}D}$ and the nilpotent element $q_i$ of order $k$.

Now we can determine the quiver $\mathcal{Q}^D$ for the underlying algebra of the
principal $2$-representation $\mathbf{P}_{\mathtt{i}}$ of $\mathscr{D}$. The vertices
of $\mathcal{Q}^D$ are given by indecomposable $D$-$D$-bimodules ${}^{\varphi_i}D$,
where $0\leq i\leq d-1$. For any $0\leq i\neq j\leq d-1$, there is exactly one arrow
from ${}^{\varphi_i}D$ to ${}^{\varphi_j}D$ and this arrow corresponds to $p_{ji}$.
For each $i$, there is exactly one arrow from ${}^{\varphi_i}D$ to ${}^{\varphi_i}D$ and
this arrow corresponds to $q_i$. We also have to impose all the obvious relations
for our generating homomorphisms:
\begin{equation}\label{e1}
q_jp_{ij}=0=p_{ij}q_i,\quad q_i^k=0,\quad p_{ij} p_{si}=0=p_{jt} p_{ij},
\end{equation}
where $0\leq i,j,s,t \leq d-1$ are such that $i\neq j, s\neq i$ and $j\neq t$.
In particular, if $k=2$, these relations just say that all paths of
length two in our quiver equal zero.

\begin{example}
{\rm
If $k=2, d=3$, then the corresponding quiver $\mathcal{Q}^D$ has three vertices given
by the indecomposable $D$-$D$-bimodules ${}^{\varphi_0}D, {}^{\varphi_1}D, {}^{\varphi_2}D$,
which we denote by $0, 1, 2$, respectively. Then $\mathcal{Q}^D$ is given as follows:
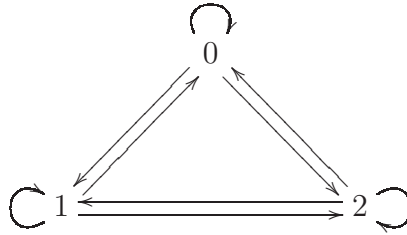
\begin{displaymath}
\xymatrix@R=1.5pc@C=1.5pc{& &0\ar@(ul,ur)[]\ar@<-.45ex>[ddll]\ar@<-.45ex>[ddrr] &&\\
&&&&\\
1\ar@(dl,ul)[]\ar@<-.5ex>[rrrr]\ar@<-.45ex>[uurr] &&&& 2\ar@(ur,dr)[]\ar@<-.45ex>[llll]\ar@<-.45ex>[uull]}
\end{displaymath}
with the relations that all paths of length two equal zero.
}
\end{example}

\begin{example}
{\rm
If $k=2, d=4$, then the corresponding quiver $\mathcal{Q}^D$ has four vertices given by indecomposable $D$-$D$-bimodules ${}^{\varphi_0}D, {}^{\varphi_1}D, {}^{\varphi_2}D, {}^{\varphi_3}D$. Using similar notation, one gets the quiver $\mathcal{Q}^D$ given as follows:
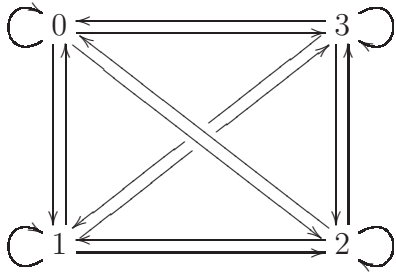
\begin{displaymath}
\xymatrix@R=1.5pc@C=1.5pc{0\ar@(dl,ul)[]\ar@<-.45ex>[rrrr]\ar@<-.45ex>[dddrrrr]\ar@<-.45ex>[ddd]&&&& 3\ar@(ur,dr)\ar@<-.45ex>[ddd]\ar@<-.45ex>[llll]\ar@<-.45ex>[dddllll]|(.49)\hole|(.51)\hole\\
&&&&\\
&&&&\\
1\ar@(dl,ul)[]\ar@<-.45ex>[rrrr]\ar@<-.45ex>[uuu]\ar@<-.45ex>[uuurrrr]|(.49)\hole|(.51)\hole &&&& 2\ar@(ur,dr)[]\ar@<-.45ex>[uuullll]\ar@<-.45ex>[llll]\ar@<-.45ex>[uuu]}
\end{displaymath}
with the relations that all paths of length two equal zero.
}
\end{example}

\subsection{Simple transitive $2$-representations for $\mathscr{D}$}
Denote by $\mathscr{I}$ the $2$-ideal in $\mathscr{D}$ generated by all $p_{ij}, q_i$,
where $0\leq i\neq j\leq d-1$. Then the quotient $2$-category $\mathscr{D}/\mathscr{I}$
is exactly the $2$-category $\mathscr{G}_G$ from \cite[Subsection~3.2]{MM5}, where
$G\cong(\mathbb{Z}_d,+)$. Let $\Theta:\mathscr{D}\to \mathscr{D}/\mathscr{I}$ be the
quotient map.

Let $\mathcal{A}$ be a fixed small category equivalent to $\mathds{C}$-mod.
For $r\vert d$, consider the subgroup $H_r$ of $\mathbb{Z}_d$ generated by $r$.
Denote by $\mathbf{V}_{r}$ the $2$-representation of $\mathscr{D}$ obtained by pulling
back, via $\Theta$, the $2$-representation $\mathbf{M}_{H_r,\mathcal{A}}$ of
$\mathscr{D}/\mathscr{I}$ from \cite[Subsection~3.2]{MM5}.

\begin{theorem}
Every simple transitive $2$-representation of $\mathscr{D}$ is
equivalent to $\mathbf{V}_{r}$ for some positive integer $r|d$.
\end{theorem}

\begin{proof}
Let $\mathbf{M}$ be a simple transitive $2$-representation of $\mathscr{D}$.
Since each $F_i$ is an equivalence, it maps an indecomposable object $X\in \mathbf{M}(\mathtt{i})$
to an indecomposable object $Y\in \mathbf{M}(\mathtt{i})$, moreover, the radical of the endomorphism
ring of $X$, being nilpotent, is mapped to the radical of the endomorphism ring of $Y$. This means that
the radical of $\mathbf{M}(\mathtt{i})$ is $\mathscr{D}$-stable. From the simple transitivity of
$\mathbf{M}$ it thus follows that $\mathbf{M}(\mathtt{i})$ is semi-simple.

Since all $p_{st}$ and  $q_i$, for $0\leq i,s,t \leq d-1, s\neq t$, are nilpotent, it follows
that $\mathbf{M}$ maps all these $2$-morphisms to zero. Therefore the representation $2$-functor
$\mathbf{M}$ factors through $\Theta$.
Thus the assertion of the theorem follows from \cite[Proposition~5]{MM5}.
\end{proof}

\subsection{The Drinfeld center of $\mathscr{D}$}

Note that each $F_i$ is isomorphic to the functor ${}^{\varphi_i}D\otimes_D{}_-$. If we identify the former with the latter, then we can interpret each equality $F_i\circ F_j= F_{i+j}$ as the corresponding $D$-$D$-bimodule isomorphism from ${}^{\varphi_i}D\otimes_D{}^{\varphi_j}D$ to
${}^{\varphi_{i+j}}D$, which sends $1\otimes 1$ to $1$.

\begin{lemma}\label{l10}
For any positive integer $s$ and any $i,j,l\in\mathbb{Z}_{\geq 0}$, where $i\not\equiv j\ (\mathrm{mod}\ d)$, we have
\begin{equation}\label{t2}
\begin{array}{rclcrcl}
\mathrm{id}_{F_l}\circ_0 q_i^s &=& q_{i+l}^s, && q_i^s\circ_0
\mathrm{id}_{F_l}&=&\zeta^{\,ls}q_{i+l}^s,\\
\mathrm{id}_{F_l}\circ_0 p_{ij} &=& p_{i+l,\,j+l},&& p_{ij}\circ_0
\mathrm{id}_{F_l}&=&\zeta^{\,l(k-1)}p_{i+l,\,j+l}.
\end{array}
\end{equation}

\end{lemma}

\begin{proof}
Note that $q_i$ is the $D$-$D$-bimodule endomorphism of ${}^{\varphi_i}D$ sending $1$ to $x$.
Therefore the morphism $\mathrm{id}_{F_l}\circ_0 q_i^s:F_l\circ F_i\to F_l \circ F_i$ is
identified with the $D$-$D$-bimodule endomorphism of ${}^{\varphi_l}D\otimes_D{}^{\varphi_i}D$
sending $1\otimes 1$ to $1\otimes x^s=(1\otimes 1)\cdot x^s$. If we identify $F_l\circ F_i$
with $F_{i+l}$ through the corresponding $D$-$D$-bimodule,
then the morphism $\mathrm{id}_{F_l}\circ_0 q_i^s:F_{i+l}\to F_{i+l}$ is exactly
identified with the endomorphism $q_{i+l}^s$ of ${}^{\varphi_{i+l}}D$ sending $1$ to $x^s$.

The morphism $q_i^s\circ_0 \mathrm{id}_{F_l}:F_i\circ F_l\to F_i \circ F_l$ is identified
with the $D$-$D$-bimodule endomorphism of ${}^{\varphi_i}D\otimes_D{}^{\varphi_l}D$ sending
$1\otimes 1$ to $x^s\otimes 1$. As
\begin{displaymath}
x^s\otimes 1=1\cdot x^s\otimes 1=1\otimes x^s\cdot 1=1\otimes \zeta^{\,ls}x^s=\zeta^{\,ls}(1\otimes 1)\cdot x^s,
\end{displaymath}
we also obtain that the morphism $q_i^s\circ_0 \mathrm{id}_{F_l}:F_{i+l}\to F_{i+l}$ is exactly
identified with the endomorphism $\zeta^{\,ls}q_{i+l}^s$ of ${}^{\varphi_{i+l}}D$ sending $1$
to $\zeta^{\,ls}x^s$.

Note that $p_{ij}$ is the $D$-$D$-bimodule homomorphism from ${}^{\varphi_i}D$ to ${}^{\varphi_j}D$ sending $1$ to $x^{k-1}$. Similarly to the case of $q_i^s$, we obtain the corresponding relations which completes the proof.
\end{proof}

\begin{remark}\label{r1}
{\rm Since $F_0\cong\mathbbm{1}_{\mathtt{i}}$, we have
\begin{displaymath}
\mathrm{id}_{F_0}\circ_0 f= f= f\circ_0 \mathrm{id}_{F_0}
\end{displaymath}
for any morphism $f\in\mathrm{Hom}_{\mathscr{D}(\mathtt{i},\mathtt{i})}(F_i,F_j)$.}
\end{remark}

For any $\mathbf{b}=(b_1,b_2,\dots,b_k)\in\mathbb{C}^k$, set
\begin{displaymath}
\mathtt{M}_{\mathbf{b}}:=
\left(
\begin{array}{ccccc}
b_1&b_2&b_3&\dots&b_k\\
0&b_1&b_2&\dots&b_{k-1}\\
0&0&b_1&\dots&b_{k-2}\\
\vdots&\vdots&\vdots&\ddots&\vdots\\
0&0&0&\dots&b_1
\end{array}
\right)\quad\text{ and }\quad
\mathtt{D}_{\mathbf{b}}:=
\left(
\begin{array}{ccccc}
b_1&0&0&\dots&0\\
0&b_2&0&\dots&0\\
0&0&b_3&\dots&0\\
\vdots&\vdots&\vdots&\ddots&\vdots\\
0&0&0&\dots&b_k
\end{array}
\right).
\end{displaymath}
If $f_i\in \mathrm{End}_{\mathscr{D}(\mathtt{i},\mathtt{i})}(F_i)$ is such that
\begin{displaymath}
f_i:=a_0\mathrm{id}_{F_i}+a_1q_i+a_2q_i^2+\cdots+a_{k-1}q_i^{k-1},
\end{displaymath}
where $\mathbf{a}=(a_0, a_1,a_2,\ldots, a_{k-1})\in \mathds{C}^k$, then the matrix of $f_i$
with respect to the {\em standard basis}  $1$, $x$, $x^2\dots$, $x^{k-1}$ of
${}^{\varphi_i}D$ is $\mathtt{M}_{\mathbf{a}}$.

Similarly, for $i\not\equiv j$ (mod $d$), if $g_{ij}\in \mathrm{Hom}_{\mathscr{D}(\mathtt{i},\mathtt{i})}(F_i,F_j)$
is such that
\begin{displaymath}
g_{ij}:=c_{k-1}p_{ij},
\end{displaymath}
where $c_{k-1}\in \mathds{C}$, then the matrix of $g_{ij}$ with respect to the standard basis $1$, $x$, $x^2\dots$, $x^{k-1}$ of
${}^{\varphi_i}D$ (resp. ${}^{\varphi_j}D$) is $\mathtt{M}_{\mathbf{c}}$, where $\mathbf{c}=(0, 0,\ldots, 0, c_{k-1})\in \mathds{C}^k$.

Set
$\boldsymbol{\zeta}:=(1,\zeta,\zeta^2,\dots,\zeta^{k-1})\in\mathds{C}^k$.
To understand the Drinfeld center $\mathcal{Z}(\mathscr{D})$, we first try to describe all pairs $(F, \Phi)$
in $\mathcal{Z}(\mathscr{D})$ and the corresponding morphism spaces in the case of indecomposable $F$.
\vspace{1cm}

\begin{theorem}\label{th28}
{\hspace{2mm}}

\begin{enumerate}[$($i$)$]
\item \label{th28.1}
All objects of the category $\mathcal{Z}(\mathscr{D})$ with indecomposable first components are, up to isomorphism,
$(F_0,\Phi^{(0)})$, where $\Phi^{(0)}$ is uniquely determined by the isomorphism
\begin{displaymath}
\Phi^{(0)}(F_1):=\mathrm{id}_{F_1}+a_1q_1+a_2q_1^2+\cdots+a_{k-1}q_1^{k-1}\in \mathrm{End}_{\mathscr{D}(\mathtt{i},\mathtt{i})}(F_1),
\end{displaymath}
and $\mathbf{a}=\mathbf{a}_{\Phi}=(1, a_1,a_2,\ldots, a_{k-1})\in \mathds{C}^k$ is such that
\begin{equation}\label{eqc1}
(\mathtt{D}_{\boldsymbol{\zeta}}\cdot\mathtt{M}_{\mathbf{a}})^d=1_k.
\end{equation}
\item \label{th28.2}
If $k\leq d$, then equality~\eqref{eqc1} holds automatically. If $k>d$, then equality~\eqref{eqc1}
is equivalent to the fact that $\mathtt{D}_{\boldsymbol{\zeta}}\cdot\mathtt{M}_{\mathbf{a}}$ is diagonalizable.
\item \label{th28.3}
For any two objects $(F_0,\Phi^{(0)}),(F_0,\Psi^{(0)})$ in 
$\mathcal{Z}(\mathscr{D})$, the corresponding morphism space $\mathrm{Hom}_{\mathcal{Z}(\mathscr{D})}((F_0,\Phi^{(0)}),(F_0,\Psi^{(0)}))$ consists of elements of the form:
\begin{equation}\label{ee1}
f=l_0\mathrm{id}_{F_0}+l_1q_0+l_2q_0^2+\cdots+l_{k-1}q_0^{k-1} \in \mathrm{End}_{\mathscr{D}(\mathtt{i},\mathtt{i})}(F_0)
\end{equation}
for some $\mathbf{l}=\mathbf{l}_{f}=(l_0,l_1,l_2,\ldots, l_{k-1})\in \mathds{C}^k$ such that
\begin{equation}\label{eqc2}
\mathtt{D}_{\boldsymbol{\zeta}}\mathtt{M}_{\mathbf{a}_{\Phi}}\mathtt{M}_{\mathbf{l}}
=\mathtt{M}_{\mathbf{l}}\mathtt{D}_{\boldsymbol{\zeta}}\mathtt{M}_{\mathbf{a}_{\Psi}}.
\end{equation}
\end{enumerate}
\end{theorem}

\begin{proof}
For each $0\leq i\leq d-1$, if $(F_i,\Phi^{(i)})$ is an object in $\mathcal{Z}(\mathscr{D})$,
then, by definition, $\Phi^{(i)}$ is a natural isomorphism from the endofunctor
$F_i\circ {}_-$ to the endofunctor ${}_-\circ F_i$ of the category $\mathscr{D}(\mathtt{i},\mathtt{i})$.
By \eqref{e5} and the fact that $F_s= F_1^s$ for any positive integer $s$, the natural isomorphism
$\Phi^{(i)}$ is uniquely determined by the isomorphism
\begin{displaymath}
\Phi^{(i)}(F_1):F_i\circ F_1\to F_1\circ F_i, \end{displaymath}
which lies in $\mathrm{End}_{\mathscr{D}(\mathtt{i},\mathtt{i})}(F_{i+1})$.
As $q_{i+1}^k=0$ and $\Phi^{(i)}(F_1)$ is an isomorphism, we may assume
\begin{displaymath}
\Phi^{(i)}(F_1)=a_0\mathrm{id}_{F_{i+1}}+a_1q_{i+1}+a_2q_{i+1}^2+\cdots+a_{k-1}q_{i+1}^{k-1},
\end{displaymath}
where $a_0\in\mathds{C}^\times$ and $a_i\in\mathds{C}$ for all other $i$.
By~\eqref{e5} and Lemma~\ref{l10}, we have
\begin{equation}\label{s1}
\begin{array}{rcl}
\Phi^{(i)}(F_s)&=&\Phi^{(i)}(\underbrace{F_1\circ F_1\circ \cdots \circ F_1}_{s \text{ times}})\\
&=&(\mathrm{id}_{F_1}\circ_0 \mathrm{id}_{F_1}\circ_0 \cdots \circ_0 \Phi^{(i)}(F_1) )\circ_1(\mathrm{id}_{F_1}\circ_0 \cdots \circ_0\Phi^{(i)}(F_1)\circ_0 \mathrm{id}_{F_1})\circ_1\\
&&\cdots\circ_1 (\Phi^{(i)}(F_1)\circ_0\cdots\circ_0 \mathrm{id}_{F_1}\circ_0 \mathrm{id}_{F_1})\\
&=&(a_0\mathrm{id}_{F_{i+s}}+a_1q_{i+s}+a_2q_{i+s}^2+\cdots+a_{k-1}q_{i+s}^{k-1})\circ_1\\
&&(a_0\mathrm{id}_{F_{i+s}}+a_1\zeta q_{i+s}+a_2\zeta^2 q_{i+s}^2+\cdots+a_{k-1}\zeta^{k-1} q_{i+s}^{k-1})\circ_1\\
&&\cdots\cdots\\
&&(a_0\mathrm{id}_{F_{i+s}}+a_1\zeta^{s-1} q_{i+s}+a_2\zeta^{2(s-1)} q_{i+s}^2+\cdots+a_{k-1}\zeta^{(k-1)(s-1)} q_{i+s}^{k-1}).
\end{array}
\end{equation}

The element $a_0\mathrm{id}_{F_{i+s}}+a_1\zeta^j q_{i+s}+a_2\zeta^{2j} q_{i+s}^2+\cdots+a_{k-1}\zeta^{(k-1)j}
q_{i+s}^{k-1}$
is given in the standard basis by the matrix
\begin{displaymath}
\left(\begin{array}{ccccc}
 a_0&\zeta^ja_1&\zeta^{2j}a_2&\dots& \zeta^{(k-1)j}a_{k-1}\\
 0&a_0&\zeta^ja_1&\dots& \zeta^{(k-2)j}a_{k-2}\\
 0&0&a_0&\dots& \zeta^{(k-3)j}a_{k-3}\\
 \vdots&\vdots&\vdots&\ddots&\vdots\\
 0&0&0&\dots&a_0 \end{array}
\right)
\end{displaymath}
which can be written as $\mathtt{D}_{\boldsymbol{\zeta}}^{-j}\cdot\mathtt{M}_{\mathbf{a}}\cdot
\mathtt{D}_{\boldsymbol{\zeta}}^j$. Therefore $\Phi^{(i)}(F_s)$ can be rewritten as follows:
\begin{equation} \label{ss1}
\mathtt{D}_{\boldsymbol{\zeta}}^{-s}\cdot(\mathtt{D}_{\boldsymbol{\zeta}}\cdot\mathtt{M}_{\mathbf{a}})^s. \end{equation}

Due to naturality of $\Phi^{(i)}$, we have the following commutative diagram
\begin{displaymath}
\xymatrix{F_i\circ F_1\ar[rr]^{\Phi^{(i)}(F_1)}\ar[d]_{\mathrm{id}_{F_i}\circ_0\, q_1}
&&F_1\circ F_i\ar[d]^{q_1\circ_0\, \mathrm{id}_{F_i}}\\
F_i\circ F_1\ar[rr]^{\Phi^{(i)}(F_1)}&&F_1\circ F_i,}
\end{displaymath}
that is, the equation $(q_1\circ_0\, \mathrm{id}_{F_i})\circ_1 \Phi^{(i)}(F_1)=\Phi^{(i)}(F_1)\circ_1(\mathrm{id}_{F_i}\circ_0\, q_1)$ holds. By Lemma~\ref{l10},
the left hand side of this equation is
\begin{displaymath}
(q_1\circ_0\, \mathrm{id}_{F_i})\circ_1 \Phi^{(i)}(F_1)=\zeta^iq_{i+1}\circ_1(a_0\mathrm{id}_{F_{i+1}}+a_1q_{i+1}+a_2q_{i+1}^2+\cdots+a_{k-1}q_{i+1}^{k-1})
\end{displaymath}
and the right hand side  of this equation is
\begin{displaymath}
\Phi^{(i)}(F_1)\circ_1(\mathrm{id}_{F_i}\circ_0\, q_1)=(a_0\mathrm{id}_{F_{i+1}}+a_1q_{i+1}+a_2q_{i+1}^2+\cdots+a_{k-1}q_{i+1}^{k-1})\circ_1 q_{i+1}.
\end{displaymath}
Comparing the coefficients of each term on both sides, we get $a_j=\zeta^i a_j$ for all
$0\leq j\leq k-2$. If $i\neq 0$, then $\zeta^i\neq1$ which implies $a_j=0$ for $0\leq j\leq k-2$.
Therefore for $0<i\leq d-1$ such natural isomorphisms $\Phi^{(i)}$ do not exist. If $i=0$, it is
clear that the left hand side coincides with the right hand side.

Now, consider a pair $(F_0,\Phi^{(0)})$ and assume that
\begin{equation}\label{e2}
\Phi^{(0)}(F_1)=a_0\mathrm{id}_{F_{1}}+a_1q_{1}+a_2q_{1}^2+\cdots+a_{k-1}q_{1}^{k-1}.
\end{equation}
Then, for each $i$, the endomorphism algebra $\mathrm{End}_{\mathscr{D}(\mathtt{i},\mathtt{i})}(F_i)$
is commutative since it is generated by $\mathrm{id}_{F_i}$ and $q_i$. For any morphism
$f\in\mathrm{Hom}_{\mathscr{D}(\mathtt{i},\mathtt{i})}(F_i,F_j)$, by Remark~\ref{r1} and the
fact that $F_t\circ F_{t'}= F_{t+t'}$, the commutativity of the following diagram
\begin{displaymath}
\xymatrix{F_0\circ F_i\ar[rr]^{\Phi^{(0)}(F_i)}\ar[d]_{\mathrm{id}_{F_0}\circ_0f}&&F_i\circ F_0\ar[d]^{f\circ_0\, \mathrm{id}_{F_0}}\\
F_0\circ F_j\ar[rr]^{\Phi^{(0)}(F_j)}&&F_j\circ F_0}
\end{displaymath}
is equivalent to the commutativity of the following diagram
\begin{equation}\label{d1}
\xymatrix{F_i\ar[rr]^{\Phi^{(0)}(F_i)}\ar[d]_{f}&&F_i\ar[d]^{f}\\
F_j\ar[rr]^{\Phi^{(0)}(F_j)}&&F_j}
\end{equation}
If $0\leq i=j\leq d-1$, the diagram~\eqref{d1} commutes as the endomorphism
algebra $\mathrm{End}_{\mathscr{D}(\mathtt{i},\mathtt{i})}(F_i)$ is commutative.
If $0\leq i\neq j\leq d-1$, by Lemma \ref{l7} we know that
$\mathrm{Hom}_{\mathscr{D}(\mathtt{i},\mathtt{i})}(F_i,F_j)=\mathds{C} p_{ij}$.
Using \eqref{e1} and \eqref{s1}, we have
\begin{displaymath}
p_{ij}\circ_1\Phi^{(0)}(F_i)=a_0^ip_{ij} \quad\text{and}\quad \Phi^{(0)}(F_j)\circ_1p_{ij}=a_0^jp_{ij}.
\end{displaymath}
The formula $p_{ij}\circ_1\Phi^{(0)}(F_i)=\Phi^{(0)}(F_j)\circ_1p_{ij}$ implies $a_0^i=a_0^j$
for all $0\leq i\neq j\leq d-1$. Since $a_0\neq 0$, we get $a_0=1$.

By Remark~\ref{re1}\,\eqref{re12} and the fact that $F_1^d= F_0$, we have
\begin{equation}\label{e6}
\Phi^{(0)}(\underbrace{F_1\circ F_1\circ\cdots\circ F_1}_{d\text{ times}})=\Phi^{(0)}(F_0)=\mathrm{id}_{F_{0}},
\end{equation}
due to \eqref{ss1}, which implies equation~\eqref{eqc1}. This completes the proof of claim~\eqref{th28.1}.

To prove claim~\eqref{th28.2}, we note that the diagonal entries of the upper triangular matrix
$\mathtt{D}_{\boldsymbol{\zeta}}\cdot\mathtt{M}_{\mathbf{a}}$ are exactly $1,\zeta,\zeta^{2},\dots,\zeta^{k-1}$.
If $d\geq k$, then all these diagonal entries are different and hence the matrix has a simple spectrum.
Since each diagonal entry (=eigenvalue) is a $d$-th root of unity, it follows that equation~\eqref{eqc1}
is an identity, for all $\mathbf{a}$. Similarly,
if $\mathtt{D}_{\boldsymbol{\zeta}}\cdot\mathtt{M}_{\mathbf{a}}$ is
diagonalizable (for any $d$ and $k$), then equation~\eqref{eqc1} is an identity. While, if
$\mathtt{D}_{\boldsymbol{\zeta}}\cdot\mathtt{M}_{\mathbf{a}}$ is not
diagonalizable, then equation~\eqref{eqc1} cannot hold. This proves claim~\eqref{th28.2}.

It remains to prove claim~\eqref{th28.3}.
Let $(F_0,\Psi^{(0)})$ be another object in $\mathcal{Z}(\mathscr{D})$, then we may assume
\begin{equation}
\Psi^{(0)}(F_1)=\mathrm{id}_{F_{1}}+a_1'q_{1}+a_2'q_{1}^2+\cdots+a_{k-1}'q_{1}^{k-1},
\end{equation}
where $\mathbf{a}_{\Psi}:=(1,a_1',a_2',\ldots,a_{k-1}')\in\mathds{C}^k$.
Now we consider the homomorphism space from $(F_0,\Phi^{(0)})$ to $(F_0,\Psi^{(0)})$ which is a subspace of $\mathrm{End}_{\mathscr{D}}(F_0)$ by definition.
If $f$ lies in this homomorphism space, assuming $f$ has the form~\eqref{ee1},  then it just need to satisfy
\begin{displaymath}
(\mathrm{id}_{F_1}\circ_0 f)\circ_1\Phi^{(0)}(F_1)=\Psi^{(0)}(F_1)\circ_1(f\circ_0 \mathrm{id}_{F_1}).
\end{displaymath}
Indeed, if this equation holds, then~\eqref{2c} automatically holds for all $F_i$ by~\eqref{s1}
and thus for any $1$-morphism $H$.
Using the matrix language and by Lemma~\ref{l10}, the above equation
is equivalent to
\begin{displaymath}
\mathtt{M}_{\mathbf{a}_{\Phi}}\mathtt{M}_{\mathbf{l}}=
\mathtt{D}_{\boldsymbol{\zeta}}^{-1} \mathtt{M}_{\mathbf{l}}\mathtt{D}_{\boldsymbol{\zeta}}
\mathtt{M}_{\mathbf{a}_{\Psi}},
\end{displaymath}
and hence to equation~\eqref{eqc2}. This completes the proof.
\end{proof}

\subsection{First example: $k=d=2$}
By Theorem~\ref{th28}, we see that indecomposable objects in the category $\mathcal{Z}(\mathscr{D})$ are
of the form $(F_0,\Phi^{(0)})$, up to isomorphism, where
\begin{displaymath}
\Phi^{(0)}(F_1):=\mathrm{id}_{F_1}+aq_1\in \mathrm{End}_{\mathscr{D}(\mathtt{i},\mathtt{i})}(F_1),\quad
\text{for }a\in \mathds{C}.
\end{displaymath}
Moreover, for any two indecomposable objects $(F_0,\Phi^{(0)})$ and $(F_0,\Psi^{(0)})$, the
corresponding homomorphism space is explicitly given by
\begin{equation}\label{eqnn32-2}
\mathrm{Hom}_{\mathcal{Z}(\mathscr{D})}((F_0,\Phi^{(0)}),(F_0,\Psi^{(0)}))\cong
\mathds{C}(\mathrm{id}_{F_0}+\frac{b-a}{2}q_0),
\end{equation}
where $\Phi^{(0)}(F_1):=\mathrm{id}_{F_1}+aq_1$ and $\Psi^{(0)}(F_1):=\mathrm{id}_{F_1}+bq_1$ with $a,b\in\mathds{C}$.

Indeed, in this case we have $\zeta=-1$. For any
$f\in\mathrm{Hom}_{\mathcal{Z}(\mathscr{D})}((F_0,\Phi^{(0)}),(F_0,\Psi^{(0)})$, we may assume
$f=l_0\,\mathrm{id}_{F_0}+l_1q_0$.
The condition~\eqref{eqc2} turns to
\begin{displaymath}
\left(\begin{array}{cc} 1 & a  \\ 0 & -1
\end{array}\right)
\left(\begin{array}{cc} l_0 & l_1  \\ 0 & l_0
\end{array}\right)=
\left(\begin{array}{cc} l_0 & l_1  \\ 0 & l_0
\end{array}\right)
\left(\begin{array}{cc} 1 & b  \\ 0 & -1
\end{array}\right).
\end{displaymath}
Thus we get $l_1=(b-a)l_0/2$ implying \eqref{eqnn32-2}.

\subsection{Second example: condition~\eqref{eqc1} for $d=2$}
We would like to discuss the condition~\eqref{eqc1} in the case $d=2$,
which reads
\begin{equation}\label{eqnn32}
\left(\begin{array}{ccccc}
1 & a_1 & a_2 & \dots & a_{k-1} \\
0 & -1 & -a_1 & \dots &
-a_{k-2}\\0 & 0 & 1 &  \dots& a_{k-3}\\
\vdots& \vdots & \vdots & \ddots & \vdots\\0& 0 &0  &\dots & (-1)^{k-1}
\end{array}\right)^2=1_k.
\end{equation}

For any integer $1\leq j\leq k-1$, the coefficient on the upper $j$-th diagonal
of the left hand side is $\displaystyle \sum_{i=0}^{j}(-1)^{i}a_ia_{j-i}$, which
must coincide  with the corresponding coefficient $0$ on the right hand side.
If $j$ is odd, then either $i$ or $j-i$ is odd for $0\leq i\leq j$. Thus the terms
$(-1)^ia_ia_{j-i}$ and $(-1)^{j-i}a_{j-i}a_i$ have different signs. Therefore the
sum $\displaystyle \sum_{i=0}^{j}(-1)^{i}a_ia_{j-i}$ always equals zero for odd $j$.
Therefore each parameter $a_{j}$ can be chosen freely for all odd $j$.
For even $j'$, the corresponding parameter $a_{j'}$ is then uniquely
determined by \eqref{eqnn32} and the choice of all parameters $a_{j}$ for odd $j$.

Finally we give a description of the whole Drinfeld center $\mathcal{Z}(\mathscr{D})$.

\begin{theorem}\label{th28'}
{\hspace{2mm}}

\begin{enumerate}[$($i$)$]
\item \label{th28'.1}
Objects of the category $\mathcal{Z}(\mathscr{D})$ are, up to isomorphism, pairs
$(F,\Phi)$, where $F$ is the direct sum of $s$ copies of $F_0$ for any positive integer $s$ and $\Phi$
is given in the standard bases by
\begin{displaymath}
\Phi(F_1):=\mathtt{M}_{\Phi}=\left(
\begin{array}{ccccc}
\mathtt{M}_{\mathbf{a}_{11}}&\mathtt{M}_{\mathbf{a}_{12}}&\mathtt{M}_{\mathbf{a}_{13}}&\dots&\mathtt{M}_{\mathbf{a}_{1s}}\\
\mathtt{M}_{\mathbf{a}_{21}}&\mathtt{M}_{\mathbf{a}_{22}}&\mathtt{M}_{\mathbf{a}_{23}}&\dots&\mathtt{M}_{\mathbf{a}_{2s}}\\
\mathtt{M}_{\mathbf{a}_{31}}&\mathtt{M}_{\mathbf{a}_{32}}&\mathtt{M}_{\mathbf{a}_{33}}&\dots&\mathtt{M}_{\mathbf{a}_{3s}}\\
\vdots&\vdots&\vdots&\ddots&\vdots\\
\mathtt{M}_{\mathbf{a}_{s1}}&\mathtt{M}_{\mathbf{a}_{s2}}&\mathtt{M}_{\mathbf{a}_{s3}}&\dots&\mathtt{M}_{\mathbf{a}_{ss}}
\end{array}
\right)\in \mathrm{End}_{\mathscr{D}(\mathtt{i},\mathtt{i})}(F),
\end{displaymath}
where $\mathbf{a}_{pp}=(1, a_1^{pp},a_2^{pp},\ldots, a_{k-1}^{pp})\in \mathds{C}^k$ and $\mathbf{a}_{pq}=(0, a_1^{pq},a_2^{pq},\ldots, a_{k-1}^{pq})\in \mathds{C}^k$ for all $1\leq p\neq q \leq s$, such that
\begin{equation}\label{eqc1'}
(\mathtt{D}_{s}\cdot\mathtt{M}_{\Phi})^d=1_{sk},
\end{equation}
where
\begin{equation}\label{di1}
\mathtt{D}_{s}=\left(
\begin{array}{ccccc}
\mathtt{D}_{\boldsymbol{\zeta}}&0&0&\dots&0\\
0&\mathtt{D}_{\boldsymbol{\zeta}}&0&\dots&0\\
0&0&\mathtt{D}_{\boldsymbol{\zeta}}&\dots&0\\
\vdots&\vdots&\vdots&\ddots&\vdots\\
0&0&0&\dots&\mathtt{D}_{\boldsymbol{\zeta}}
\end{array}
\right)_{sk\times sk}.
\end{equation}
\item \label{th28'.2}
The condition~\eqref{eqc1'} is equivalent to the fact that $\mathtt{D}_{s}\cdot\mathtt{M}_{\Phi}$ is diagonalizable.
\item \label{th28'.3}
For any two objects $(F,\Phi),(G,\Psi)$ in $\mathcal{Z}(\mathscr{D})$, where
\begin{displaymath}
F:=\underbrace{F_0\oplus F_0\oplus \cdots \oplus F_0}_{s \text{ times}} \quad\text{and}\quad G:=\underbrace{F_0\oplus F_0\oplus \cdots \oplus F_0}_{t \text{ times}}\quad\text{for some } s,t \in \mathbb{Z}_{>0},
\end{displaymath}
 the corresponding morphism space $\mathrm{Hom}_{\mathcal{Z}(\mathscr{D})}((F,\Phi),(G,\Psi))$ consists of elements
 given in the standard bases by the following matrices:
\begin{equation}\label{ee1'}
\mathtt{M}_f=\left(
\begin{array}{ccccc}
\mathtt{M}_{\mathbf{l}_{11}}&\mathtt{M}_{\mathbf{l}_{12}}&\mathtt{M}_{\mathbf{l}_{13}}&\dots&\mathtt{M}_{\mathbf{l}_{1t}}\\
\mathtt{M}_{\mathbf{l}_{21}}&\mathtt{M}_{\mathbf{l}_{22}}&\mathtt{M}_{\mathbf{l}_{23}}&\dots&\mathtt{M}_{\mathbf{l}_{2t}}\\
\mathtt{M}_{\mathbf{l}_{31}}&\mathtt{M}_{\mathbf{l}_{32}}&\mathtt{M}_{\mathbf{l}_{33}}&\dots&\mathtt{M}_{\mathbf{l}_{3t}}\\
\vdots&\vdots&\vdots&\ddots&\vdots\\
\mathtt{M}_{\mathbf{l}_{s1}}&\mathtt{M}_{\mathbf{l}_{s2}}&\mathtt{M}_{\mathbf{l}_{s3}}&\dots&\mathtt{M}_{\mathbf{l}_{st}}
\end{array}
\right)
\end{equation}
for some $\mathbf{l}_{pq}:=(l_0^{pq}, l_1^{pq}, l_2^{pq},\ldots, l_{k-1}^{pq})\in \mathds{C}^k$, where $1\leq p\leq s, 1\leq q\leq t$, such that
\begin{equation}\label{eqc2'}
\mathtt{D}_{s}\mathtt{M}_{\Phi}\mathtt{M}_{f}=
\mathtt{M}_{f}\mathtt{D}_{t}\mathtt{M}_{\Psi}.
\end{equation}
\end{enumerate}
\end{theorem}

\begin{proof}
Let $(F,\Phi)$ be an object in $\mathcal{Z}(\mathscr{D})$. Assume that
\begin{displaymath}
F:=F_{i_1}\oplus F_{i_2}\oplus \cdots \oplus F_{i_s}
\end{displaymath}
where $0\leq i_1\leq i_2\leq\ldots\leq i_s\leq d-1$ and  $s $ is a positive integer.
Then, by definition, $\Phi$ is a natural isomorphism from the endofunctor
$F\circ {}_-$ to the endofunctor ${}_-\circ F$ of the category $\mathscr{D}(\mathtt{i},\mathtt{i})$.
By \eqref{e5} and the fact that $F_i= F_1^i$, the natural isomorphism
$\Phi$ is uniquely determined by the isomorphism $\Phi(F_1)$ in
$\mathrm{End}_{\mathscr{D}(\mathtt{i},\mathtt{i})}(F_{i_1+1}\oplus F_{i_2+1}\oplus \cdots \oplus F_{i_s+1})$.

With respect to the standard basis, we may assume that $\Phi(F_1)$ is given by
\begin{displaymath}
\mathtt{M}_{\Phi}=\left(
\begin{array}{ccccc}
\mathtt{M}_{\mathbf{a}_{11}}&\mathtt{M}_{\mathbf{a}_{12}}&\mathtt{M}_{\mathbf{a}_{13}}&\dots&\mathtt{M}_{\mathbf{a}_{1s}}\\
\mathtt{M}_{\mathbf{a}_{21}}&\mathtt{M}_{\mathbf{a}_{22}}&\mathtt{M}_{\mathbf{a}_{23}}&\dots&\mathtt{M}_{\mathbf{a}_{2s}}\\
\mathtt{M}_{\mathbf{a}_{31}}&\mathtt{M}_{\mathbf{a}_{32}}&\mathtt{M}_{\mathbf{a}_{33}}&\dots&\mathtt{M}_{\mathbf{a}_{3s}}\\
\vdots&\vdots&\vdots&\ddots&\vdots\\
\mathtt{M}_{\mathbf{a}_{s1}}&\mathtt{M}_{\mathbf{a}_{s2}}&\mathtt{M}_{\mathbf{a}_{s3}}&\dots&\mathtt{M}_{\mathbf{a}_{ss}}
\end{array}
\right)
\end{displaymath}
where $\mathbf{a}_{pq}=(a_0^{pq}, a_1^{pq},a_2^{pq},\ldots, a_{k-1}^{pq})\in \mathds{C}^k$ for all $1\leq p,q \leq s$.

By~\eqref{e5} and Lemma~\ref{l10}, similarly to the proof of Theorem~\ref{th28}, for any positive integer $i$ we can write $\Phi(F_i)$ as follows:
\begin{equation} \label{33s}
\mathtt{D}_{s}^{-i}\cdot(\mathtt{D}_{s}\cdot \mathtt{M}_{\Phi})^i,
 \end{equation}
where $\mathtt{D}_s$ is given by equation~\eqref{di1}.

Applying the naturality of $\Phi$ to the $2$-morphism $p_{ij}$, we have the following commutative diagram
\begin{equation}\label{cm1}
\xymatrix{F_{i_1+i}\oplus F_{i_2+i}\oplus \cdots \oplus F_{i_s+i}
\ar[rr]^{\Phi(F_i)}\ar[d]_{\mathrm{id}_{F}\circ_0\, p_{ij}}
&&F_{i_1+i}\oplus F_{i_2+i}\oplus \cdots \oplus F_{i_s+i}\ar[d]^{p_{ij}\circ_0\,
\mathrm{id}_{F}}\\
F_{i_1+j}\oplus F_{i_2+j}\oplus \cdots \oplus F_{i_s+j}\ar[rr]^{\Phi(F_j)}&&
F_{i_1+j}\oplus F_{i_2+j}\oplus \cdots \oplus F_{i_s+j}}
\end{equation}
Due to~\eqref{t2} and~\eqref{33s}, the matrix translation of this commutative diagram is
\begin{multline}\label{di2}
\mathtt{D}_{s}^{-i}\cdot(\mathtt{D}_{s}\cdot \mathtt{M}_{\Phi})^i \cdot
\left(
\begin{array}{ccccc}
\zeta^{i_1(k-1)}\mathtt{M}_{\underline{\boldsymbol{1}}_{k-1}}&0&0&\dots&0\\
0&\zeta^{i_2(k-1)}\mathtt{M}_{\underline{\boldsymbol{1}}_{k-1}}&0&\dots&0\\
0&0&\zeta^{i_3(k-1)}\mathtt{M}_{\underline{\boldsymbol{1}}_{k-1}}&\dots&0\\
\vdots&\vdots&\vdots&\ddots&\vdots\\
0&0&0&\dots&\zeta^{i_s(k-1)}\mathtt{M}_{\underline{\boldsymbol{1}}_{k-1}}
\end{array}
\right)_{sk\times sk}  \\
=\left(
 \begin{array}{ccccc}
 \mathtt{M}_{\underline{\boldsymbol{1}}_{k-1}}&0&0&\dots&0\\
 0&\mathtt{M}_{\underline{\boldsymbol{1}}_{k-1}}&0&\dots&0\\
 0&0&\mathtt{M}_{\underline{\boldsymbol{1}}_{k-1}}&\dots&0\\
 \vdots&\vdots&\vdots&\ddots&\vdots\\
 0&0&0&\dots&\mathtt{M}_{\underline{\boldsymbol{1}}_{k-1}}
 \end{array}
 \right)_{sk\times sk}\cdot \mathtt{D}_{s}^{-j}\cdot(\mathtt{D}_{s}\cdot \mathtt{M}_{\Phi})^j
\end{multline}
where $\underline{\boldsymbol{1}}_{k-1}=(0,0,\ldots,0,1)$.

When $i=1, j=0$,  equation~\eqref{di2} implies
\begin{equation}
a_0^{pp}=\zeta^{-i_p(k-1)}\neq 0  \text{ for all } 1\leq p\leq s \quad\text{ and }\quad
a_0^{pq}=0 \text{ for all } 1\leq p\neq q\leq s.\label{bi3}
\end{equation}

Applying the naturality of $\Phi$ to the $2$-morphism $q_1$, we have the following commutative diagram
\begin{equation}\label{cm2}
\xymatrix{F_{i_1+1}\oplus F_{i_2+1}\oplus \cdots \oplus F_{i_s+1}
\ar[rr]^{\Phi(F_1)}\ar[d]_{\mathrm{id}_{F}\circ_0\, q_{1}}
&&F_{i_1+1}\oplus F_{i_2+1}\oplus \cdots \oplus F_{i_s+1}\ar[d]^{q_{1}\circ_0\,
\mathrm{id}_{F}}\\
F_{i_1+1}\oplus F_{i_2+1}\oplus \cdots \oplus F_{i_s+1}\ar[rr]^{\Phi(F_1)}&&
F_{i_1+1}\oplus F_{i_2+1}\oplus \cdots \oplus F_{i_s+1}}
\end{equation}
Using~\eqref{t2} and~\eqref{33s}, this gives
\begin{displaymath}
 \mathtt{M}_{\Phi}\cdot
 \left(
 \begin{array}{ccccc}
 \zeta^{i_1}\mathtt{M}_{\underline{\boldsymbol{1}}_{1}}&0&0&\dots&0\\
 0&\zeta^{i_2}\mathtt{M}_{\underline{\boldsymbol{1}}_{1}}&0&\dots&0\\
 0&0&\zeta^{i_3}\mathtt{M}_{\underline{\boldsymbol{1}}_{1}}&\dots&0\\
 \vdots&\vdots&\vdots&\ddots&\vdots\\
 0&0&0&\dots&\zeta^{i_s}\mathtt{M}_{\underline{\boldsymbol{1}}_{1}}
 \end{array}
 \right)_{sk\times sk}=
 \left(
 \begin{array}{ccccc}
 \mathtt{M}_{\underline{\boldsymbol{1}}_{1}}&0&0&\dots&0\\
 0&\mathtt{M}_{\underline{\boldsymbol{1}}_{1}}&0&\dots&0\\
 0&0&\mathtt{M}_{\underline{\boldsymbol{1}}_{1}}&\dots&0\\
 \vdots&\vdots&\vdots&\ddots&\vdots\\
 0&0&0&\dots&\mathtt{M}_{\underline{\boldsymbol{1}}_{1}}
 \end{array}
 \right)_{sk\times sk}\cdot \mathtt{M}_{\Phi}
 \end{displaymath}
 where $\underline{\boldsymbol{1}}_{1}=(0,1,0,\ldots,0,0)$. The above equation yields
 \begin{displaymath}
 a_t^{pq}\zeta^{i_q}=a_t^{pq} \text{ for all }1\leq p,q \leq s,\ 0\leq t\leq k-2.
 \end{displaymath}
 When $p=q$, using~\eqref{bi3}, we obtain $\zeta^{i_p}=1$ for all $1\leq p\leq s$ and hence $d|i_p$.
 Therefore we get $F_{i_p}=F_0$ for all $1\leq p\leq s$ and therefore $F=\underbrace{F_0\oplus
 F_0\oplus\cdots \oplus F_0}_{s \text{ times}}$. Moreover, it follows from~\eqref{bi3} that $a_0^{pp}=1$
for all $1\leq p\leq s$ and thus equation~\eqref{di2} holds automatically.
In fact, by~\eqref{e5}, we see that the naturality of $\Phi$ follows from the two
commutative diagrams~\eqref{cm1} and~\eqref{cm2}.

Note that $F_d=F_0$. Thus we have $\Phi(F_d)=\Phi(F_0)=\mathrm{id}_{F}$
due to~\eqref{33s}, that is, we have \eqref{eqc1'}. The latter identity
also implies invertibility of $\mathtt{M}_{\Phi}$ and completes the proof of claim~\eqref{th28'.1}.

Since the algebra of upper triangular matrixes with zero diagonal is nilpotent,
the matrix $\mathtt{M}_{\Phi}-1$ is nilpotent.
This means that $1$ is the only eigenvalue of $\mathtt{M}_{\Phi}$.
Thus all eigenvalues of $\mathtt{D}_{s}\cdot \mathtt{M}_{\Phi}$ are $1,
\zeta,\zeta^{2},\ldots,\zeta^{k-1}$,
each of which has multiplicity $s$. This implies claim~\eqref{th28'.2}.

It remains to prove claim~\eqref{th28'.3}.
Let $(G,\Psi)$ be another object in $\mathcal{Z}(\mathscr{D})$,
where $G$ is the direct sum of $t\in\mathbbm{Z}_{>0}$ copies of $F_0$.
Now we consider the homomorphism space $\mathrm{Hom}_{\mathcal{Z}(\mathscr{D})}((F,\Phi),(G,\Psi))$
which is a subspace of $\mathrm{Hom}_{\mathscr{D}}(F,G)$ by definition.
If $f$ lies in this homomorphism space, we may assume that $f$ has the form~\eqref{ee1'}.
Then it is sufficient to prove
\begin{displaymath}
  \xymatrix{{\overbrace{F_{1}\oplus F_{1}\oplus \cdots \oplus F_{1}}^{s \text{ times}}}
  \ar[rr]^{\Phi(F_1)}\ar[d]_{f\circ_0\, \mathrm{id}_{F_{1}}}
  &&{\overbrace{F_{1}\oplus F_{1}\oplus \cdots \oplus F_{1}}^{s \text{ times}}} \ar[d]^{\mathrm{id}_{F_1}\circ_0\,
  f}\\
  {\underbrace{F_{1}\oplus F_{1}\oplus \cdots \oplus F_{1}}_{t \text{ times}}}
  \ar[rr]^{\Psi(F_1)}&& {\underbrace{F_{1}\oplus F_{1}\oplus \cdots \oplus F_{1}}_{t \text{ times}}}}
 \end{displaymath}
 If this diagram commutes, then~\eqref{2c} automatically holds for all $F_i$ by~\eqref{e5}
 and thus for any $1$-morphism $H$.
 Using the matrix language and by Lemma~\ref{l10}, the above commutative diagram
 is equivalent to
 \begin{displaymath}
\mathtt{M}_{\Phi}\mathtt{M}_{f}=
\mathtt{D}_{s}^{-1}\mathtt{M}_{f}\mathtt{D}_{t}\mathtt{M}_{\Psi},
\end{displaymath}
and hence to equation~\eqref{eqc2'}. This completes the proof.
\end{proof}

\noindent
Department of Mathematics, East China Normal University, Minhang District,
Dong Chuan Road 500, Shanghai, 200241, PR China,
E-mail address: {\tt scropure\symbol{64}126.com}

\end{document}